\documentclass{article}

\usepackage{fullpage}
\usepackage{authblk} 
\usepackage[colorlinks=true, allcolors=blue]{hyperref}
\usepackage{enumitem}
\setlist[itemize]{itemsep=0pt}
\setlist[enumerate]{itemsep=0pt}
\newcommand*{\textitplusparen}[1]{\textit{#1:}}
\setlist[description]{itemsep=0pt, font=\normalfont\textitplusparen}
\usepackage{calc} 

\usepackage[nocompress]{cite}

\usepackage{amsmath,amssymb}
\allowdisplaybreaks 

\newcommand\R{\mathbb{R}}
\newcommand\esp[1]{\mathbb{E}\left[#1\right]}
\newcommand\expo{\mathrm{Exp}}
\newcommand\e{\mathrm{e}}

\newcommand\A{{\mathcal A}}
\newcommand\C{{\mathcal C}}
\newcommand\D{{\mathcal D}}
\newcommand\E{{\mathcal E}}
\newcommand\I{{\mathcal I}}
\newcommand\J{{\mathcal J}}
\newcommand\K{{\mathcal K}}
\newcommand\V{{\mathcal V}}
\newcommand\Z{{\mathcal Z}}

\newcommand\ind{{\mathbb I}}
\newcommand\inde{\ind_0}
\newcommand\link{\sim}
\newcommand\nlink{\nsim}

\usepackage{amsthm}

\newtheorem{theorem}{Theorem}
\newtheorem{lemma}[theorem]{Lemma}

\newtheorem{proposition}[theorem]{Proposition}

\theoremstyle{definition}
\newtheorem{definition}[theorem]{Definition}
\newtheorem{assumption}[theorem]{Assumption}

\theoremstyle{remark}

\newtheorem*{remark*}{Remark}

\usepackage{xcolor}
\definecolor{red}{RGB}{255,0,0}
\definecolor{blue}{rgb}{0.0, 0.4, 0.65}
\definecolor{orange}{RGB}{253,188,64}
\definecolor{green}{RGB}{154,205,50}
\definecolor{orangeplot}{RGB}{245,138,42}
\definecolor{greenplot}{RGB}{0,200,0}

\usepackage{graphics,graphicx,subfig}
\graphicspath{{figures/}}

\usepackage{tikz,tikz-3dplot}
\usetikzlibrary{calc,arrows,shapes,snakes,backgrounds}

\tikzset{
	class/.style = {draw, minimum size=.6cm, fill=blue!20},
	smallclass/.style = {class, minimum size=.6cm},
	server/.style = {draw, circle, minimum size=.7cm, fill=blue!20},
	fcfs/.style={
		draw,
		rectangle split,
		rectangle split parts=#1,
		rectangle split horizontal,
		rectangle split empty part width=.44cm,
		rectangle split empty part height=.6cm,
		inner xsep=0cm,
		inner ysep=0cm,
	},
}

\usepackage{pgfplots,pgfplotstable}
\pgfplotsset{compat=1.8}
\usepgfplotslibrary{fillbetween}

\pgfplotsset{
	table/col sep = {comma},
	table/search path = {./},
	defaultplot/.style={
		xlabel near ticks, ylabel near ticks,
		xtick style={draw=none}, ytick style={draw=none},
		every axis plot/.append style={
			thick,
		},
	},
}

\pgfplotsset{
	loadplotstyle/.style={defaultplot,
		xlabel={Load $\rho$ of class~1},
		xmin=0, ymin=0,
		grid=major,
		legend pos=north west,
		legend style={
			cells={anchor=west, align=left},
			at={(0, 1.25)},
			/tikz/every even column/.append style={column sep=0.4cm},
		},
		legend columns=10,
		legend cell align={left},
		width=.48\linewidth, height=.3\linewidth,
	},
	probaplotstyle/.style={loadplotstyle,
		ylabel={Waiting probability},
		xmin=0, xmax=1,
		ymin=0, ymax=1,
	},
	meanplotstyle/.style={loadplotstyle,
		ylabel={Mean matching time},
		xmin=0, xmax=1,
		ymin=0, ymax=50,
		restrict y to domain=0:100,
	},
}

\usepackage{filecontents}

\begin{document}
	
	\title{Stochastic non-bipartite matching models and~order-independent~loss~queues%
		\footnote{Author version of the paper available at \url{https://doi.org/10.1080/15326349.2021.1962352}.}}
	\author{C\'eline Comte}
	\affil{Eindhoven University of Technology,
		The Netherlands \\
		\href{mailto:c.m.comte@tue.nl}{c.m.comte@tue.nl}%
	}
	\date{\vspace{-1cm}}
	
	\maketitle
	
	\begin{abstract}
		\noindent
		The problem of appropriately matching items subject to compatibility constraints arises in a number of important applications. While most of the literature on matching theory focuses on a static setting with a fixed number of items, several recent works incorporated time by considering a stochastic model in which items of different classes arrive according to independent Poisson processes and assignment constraints are described by an undirected non-bipartite graph on the classes. In this paper, we prove that the Markov process associated with this model has the same transition diagram as in a product-form queueing model called an order-independent loss queue. This allows us to adapt existing results on order-independent (loss) queues to stochastic matching models and, in particular, to derive closed-form expressions for several performance metrics, like the waiting probability or the mean matching time, that can be implemented using dynamic programming. Both these formulas and the numerical results that they allow us to derive are used to gain insight into the impact of parameters on performance. In particular, we characterize performance in a so-called heavy-traffic regime in which the number of items of a subset of the classes goes to infinity while items of other classes become scarce.
		\\[.1cm]
		\textbf{Keywords:} Stochastic matching model,
		order-independent queue,
		product-form stationary distribution,
		dynamic programming,
		heavy-traffic regime.
	\end{abstract}

	\section{Introduction}
	
	Matching items with one another
	in such a way that an individual
	item's preferences are met
	is a well-known problem
	in mathematics, computer science, and economics.
	Classical applications include
	assigning users to servers
	in distributed computing systems,
	joining components in assembly systems,
	matching passengers to drivers
	in ride-sharing applications,
	and assigning questions to experts
	in question-and-answer websites.
	
	\begin{figure}[ht]
		\centering
		\begin{tikzpicture}
			\def\delta{1.2cm}
			
			\node[class, fill=green!60] (1) {1};
			\node[class, fill=red!15] (2)
			at ($(1)-(0,\delta)$) {2};
			\node[class, fill=yellow!60] (3)
			at ($(1)!.5!(2)+(\delta,0)$) {3};
			\node[class, fill=orange!60] (4)
			at ($(3)+(1.1*\delta,0)$) {4};
			
			\draw[-] (3) -- (1) -- (2) -- (3) -- (4);
			
			\node at ($(1)-(1.9cm,0)$) {};
			\node at ($(4)+(1.9cm,0)$) {};
		\end{tikzpicture}
		\caption{An undirected non-bipartite
			compatibility graph.}
		\label{fig:toy-graph}
	\end{figure}
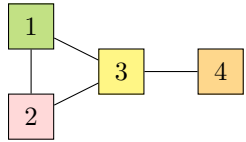
	
	While most of the literature on matching theory
	focuses on a \emph{static} setting
	in which a fixed set of items has to be matched,
	applications to modern
	large-scale service systems
	require that we also consider the time dimension.
	In assembly systems for instance,
	components of different types
	arrive at random instants
	as they are manufactured,
	and the primary objective is to have
	sufficient balance between components
	to maximize throughput
	while also minimizing the number
	of components to be stored.
	This observation gave rise to models
	that can be seen as stochastic generalizations
	of classical matching problems.
	In these stochastic matching models,
	items of different classes
	arrive according to independent Poisson processes.
	Compatibility constraints between items
	are described by an undirected graph on the classes,
	so that there is an edge between two classes
	if items of these classes
	can be matched with one another.
	In the example of \figurename~\ref{fig:toy-graph},
	class-$3$ items can be matched
	with items of classes~$1$, $2$ and~$4$,
	and items of classes~$1$ and $2$
	can also be matched with one another.
	A \emph{matching policy} is a decision rule
	that specifies, whenever an item arrives,
	to which compatible waiting unmatched item,
	if any, it is matched.
	The primary objective of a matching policy
	is to make the system \emph{stable}
	in the sense that
	the stochastic process describing
	the evolution of the number
	of unmatched items over time is ergodic.
	The secondary objective
	is to optimize performance metrics
	like the mean matching time, defined as
	the average amount of time
	between the arrival of an item
	and its matching with another item.
	
	Two variants of the stochastic matching model
	have been studied in the literature.
	The stochastic \emph{bipartite} matching model
	was introduced in~\cite{CKW09}
	and further studied
	in~\cite{AW12,BGM13,ABMW17,AKRW18}.
	As the name suggests,
	this model assumes that
	the compatibility graph is bipartite.
	One can verify that such a model
	is always unstable if items arrive
	one by one at random instants
	(as the corresponding Markov process is
	either transient or null recurrent).
	Stability in this model is made possible by imposing
	that items always arrive in pairs,
	so that the classes of items in a pair
	belong to the two parts of the graph.
	The stochastic \emph{non-bipartite} matching model,
	which we consider in this paper,
	was introduced in \cite{MM16}
	and further studied
	in \cite{CDFB20,MBM21,BMMR20}.
	In this model,
	items arrive one at a time,
	and stability is made possible
	by imposing the presence of odd cycles
	in the compatibility graph.
	Mairesse and Moyal~\cite{MM16} consider
	matching policies
	that are \emph{admissible} in the sense that
	the matching decision is
	based solely on the class of the incoming item
	and the sequence, in order of arrival,
	of classes of unmatched items.
	This work characterizes
	the maximum stability region, that is,
	the set of vectors of per-class arrival rates
	under which there exists an admissible policy that makes the system stable.
	The focus of~\cite{MBM21,CDFB20,BMMR20}
	is on an admissible policy
	called \emph{first-come-first-matched};
	this is also our focus.
	As the name suggests, each incoming item
	is matched with the compatible item
	that has been waiting the longest, if any.
	In addition to being perceived as fair,
	this policy has the advantage of
	maximizing the stability region
	and leading to tractable results,
	as we will develop in this paper.
	The rationale behind this policy,
	which gives an intuition for its maximum stability,
	is as follows:
	if an item has been waiting longer than another,
	it is likely that this item
	is compatible with classes that are relatively scarcer,
	so that matching this item if an opportunity arises
	is a good heuristic to preserve stability.
	It is shown in~\cite{MBM21} that
	this matching policy is
	indeed maximally stable
	and leads to a product-form
	stationary distribution.
	This work also derives
	a closed-form expression
	for the normalization constant
	of this stationary distribution.
	A similar closed-form expression
	for the mean matching time
	is proposed in \cite{CDFB20},
	which applies this result to study
	the impact of the compatibility graph on performance.
	The authors observe that adding edges
	sometimes has a negative impact on performance,
	which is akin to Braess's paradox in road networks.
	Lastly, the recent work \cite{BMMR20}
	adapts some of the above results
	to a variant of the model
	in which the items of a given class
	can also be matched with one another.
	
	Several recent papers~\cite{AW12,ABMW17,AKRW18,GR20} have
	pointed out the similarity between
	the stationary distribution
	of stochastic matching models
	under this first-come-first-matched policy
	and that of order-independent queues~\cite{BKK95},
	a product-form queueing model
	that recently gained momentum
	in the queueing-theory community.
	Our first main contribution
	consists of formally proving
	that a stochastic non-bipartite matching model
	is an order-independent \emph{loss} queue,
	a loss variant of order-independent queues
	introduced in~\cite{BK96}.
	This allows us to adapt
	known results on order-independent queues
	to stochastic non-bipartite matching models,
	and in particular to provide simpler proofs
	for the product-form and stability results
	derived in \cite{MBM21}.
	The relation with order-independent loss queues
	also allows us to extend these results
	to several variants of the model,
	some of which were already considered in the literature.
	Our second main contribution
	consists of closed-form expressions
	for several performance metrics
	like the waiting probability
	and the mean matching time.
	Some of these expressions
	can be seen as recursive variants
	of those derived in \cite{MBM21,CDFB20},
	and are adapted from the works \cite{SV15,C19,CBL21}
	which focus on order-independent queues
	and related queueing models.
	The recursive form of these expressions
	allows us to avoid duplicate calculations
	by using dynamic programming,
	which significantly reduces
	the complexity of the calculations.
	In addition to their numerical applications,
	these formulas allow us to gain intuition
	on the impact of the parameters on performance,
	and in particular to formalize
	what it means to have a \emph{balanced} system.
	Along the same lines,
	our third and last contribution
	is a performance analysis
	in a scaling regime, called \emph{heavy traffic}
	by analogy with the scaling regime
	of the same name in queueing theory,
	in which the vector of per-class arrival rates
	tends to a boundary of the stability region.
	This result can be seen as
	a generalization of the
	heavy-traffic result of~\cite{CBL21},
	which focuses on order-independent queues,
	although the conclusion is different
	and illustrates the subtle difference
	between stochastic matching models
	and conventional queueing models.
	
	The paper is organized as follows.
	Section~\ref{sec:model}
	describes a continuous-time variant
	of the stochastic non-bipartite matching model
	proposed in~\cite{MM16}
	and introduces useful notation.
	In Section~\ref{subsec:oi},
	we prove that
	this stochastic matching model
	is an order-independent loss queue.
	This result is applied
	in the rest of Section~\ref{sec:oi}
	and in Section~\ref{sec:perf}
	to translate known results
	for order-independent queues
	to the stochastic matching model
	and to generalize them.
	In particular, Sections~\ref{subsec:steady}
	and \ref{subsec:stability}
	give simpler proofs of
	the product-form result
	and stability condition of~\cite{MBM21},
	while in Section~\ref{sec:perf}
	we derive closed-form expressions
	for several performance metrics.
	These results are applied in Section~\ref{sec:app}
	to gain insight into the impact
	of parameters on performance.
	More specifically, we propose several heuristics
	to minimize the mean matching time
	and study what happens when
	part of the system becomes unstable.
	Numerical results in Section~\ref{sec:num}
	allow us to gain further insight
	into the dynamics
	of stochastic matching models.
	Section~\ref{sec:ccl} concludes the paper.

	\section{Stochastic non-bipartite matching model} \label{sec:model}
	
	We first define a continuous-time variant
	of the stochastic non-bipartite matching model
	introduced in \cite{MM16}.
	
	\subsection{Continuous-time model} \label{subsec:model}
	
	Let $\V = \{1, 2, \ldots, N\}$
	denote a finite set of item classes
	and consider an undirected
	connected non-bipartite
	graph without loops
	with vertex set $\V$.
	This graph will be called the
	\emph{compatibility graph} of the matching model
	because it will eventually describe
	the matching compatibilities between items.
	An example
	is shown in \figurename~\ref{fig:toy-graph}.
	For each $i, j \in \V$,
	we write $i \link j$ if there is an edge
	between classes~$i$ and $j$
	in the compatibility graph
	and $i \nlink j$ otherwise.
	For each $i \in \V$,
	we also let $\E_i \subseteq \V$
	denote the set of neighbors of class~$i$
	in the compatibility graph.
	With a slight abuse of notation,
	for each set
	$\A \subseteq \V$ of classes,
	we also let
	$\E(\A) = \bigcup_{i \in \A} \E_i$
	denote the set of classes
	that are neighbors
	of at least one class in $\A$.
	Our assumption
	that the compatibility graph is non-bipartite
	is not necessary for the soundness of the model,
	but we will see in Section~\ref{subsec:stability}
	that it is required for stability.
	
	For each $i \in \V$,
	class-$i$ items enter the system
	according to an independent
	Poisson process with rate~$\alpha_i > 0$.
	As we will see in Section~\ref{subsec:steady},
	we can (and will) assume without loss of generality
	that the vector $\alpha = (\alpha_1, \alpha_2, \ldots, \alpha_N)$
	is normalized
	in the sense that $\sum_{i \in \V} \alpha_i = 1$.
	For each subset $\A \subseteq \V$ of classes,
	we also let $\alpha(\A) = \sum_{i \in \A} \alpha_i$
	denote the overall arrival rate
	of the classes in $\A$.
	All unmatched items are queued
	in a buffer in their arrival order.
	In \figurename~\ref{fig:toy-state}
	for instance,
	unmatched items are ordered
	from the oldest on the left
	to the newest on the right,
	and the number written on each item
	represents its class.
	Each incoming item is matched with
	an item of a compatible class
	present in the buffer, if any,
	in which case both items
	disappear immediately;
	otherwise, the incoming item
	is added at the end of the buffer.
	If an incoming item
	finds several compatible items
	in the buffer,
	it is matched with the compatible item
	that has been in the buffer the longest.
	This matching policy is called
	\emph{first-come-first-matched}
	and was studied in \cite{MM16,MBM21,CDFB20,BMMR20}.
	Note that this matching policy
	can be applied
	with a more decentralized approach
	in which each item is only aware
	of the arrival order
	of its \emph{compatible} items.
	Such a decentralized approach
	leads to the same dynamics and therefore
	to the same performance measures.
	
	\begin{figure}[ht]
		\centering
		\begin{tikzpicture}
			\node[fcfs=7,
			rectangle split part fill
			= {orange!60, orange!60, green!60, orange!60, green!60, white},
			] (queue) {};
			\fill[white] ([xshift=\pgflinewidth+1pt,
			yshift=-\pgflinewidth+1pt]queue.north east)
			rectangle ([xshift=-\pgflinewidth-4pt,
			yshift=\pgflinewidth-1pt]queue.south east);
			\fill[white] ([xshift=\pgflinewidth+1pt,
			yshift=-\pgflinewidth-.01pt]queue.north east)
			rectangle ([xshift=-\pgflinewidth-5pt,
			yshift=\pgflinewidth+.01pt]queue.south east);
			\node[anchor=north, yshift=.1cm]
			at ($(queue.south)-(.1cm,0)$)
			{\strut Buffer in state $c = (4,4,1,4,1)$};
			
			\node at
			($(queue.one south)!.5!(queue.one north)$) {4};
			\node at
			($(queue.two south)!.5!(queue.two north)$) {4};
			\node at
			($(queue.three south)!.5!(queue.three north)$) {1};
			\node at
			($(queue.four south)!.5!(queue.four north)$) {4};
			\node at
			($(queue.five south)!.5!(queue.five north)$) {1};
			
			\node[class, fill=red!15, anchor=south east] (new)
			at ($(queue.north west)+(0,.6cm)$) {2};
			\node[anchor=south, yshift=-.1cm]
			at (new.north) {\strut New item};
			\draw[->] ($(new.west)-(.6cm,0)$)
			-- ($(new.west)-(.05cm,0)$);
			\draw[->] ($(new.south)-(0,.05cm)$)
			|- ($(queue.three north)+(0,.3cm)$)
			-- ($(queue.three north)+(0,.05cm)$);
			
			\draw[-] ($(queue.three)+(.6cm,.6cm)$)
			-- ($(queue.three)$);
			\draw[-] ($(queue.three)+(.6cm,0)$)
			-- ($(queue.three)+(0,.6cm)$);
			
			\node at ($(new.west)-(2.4cm,0)$) {};
			\node at ($(queue.east)+(2cm,0)$) {};
		\end{tikzpicture}
		\caption{%
			Unmatched items are ordered in the buffer
				from the oldest item on the left
				to the newest on the right.
				In particular, in the state
				$c = (4, 4, 1, 4, 1)$ depicted in the picture,
				the oldest item is of class~$4$
				and the newest is of class~$1$.
			With the compatibility graph
			of~\figurename~\ref{fig:toy-graph},
			if a class-2 item arrives
			while the system is in this state,
			this item scans the buffer
			from left to right until it finds
			a compatible item, and is therefore matched
			with the oldest class-1 item.
			The new system state after this transition
			is~$d = (4,4,4,1)$.
			\label{fig:toy-state}
		}
	\end{figure}
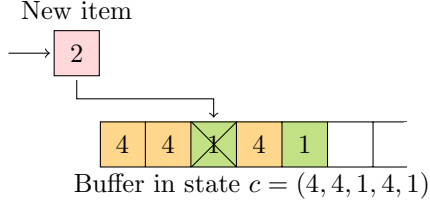
	
	\subsection{Markov process} \label{subsec:ctmc}
	
	We now define
	a Markov process
	that describes the evolution
	of the continuous-time stochastic
	matching model that we have just introduced.
	
	\paragraph{Markovian state}
	
	We describe the system state
	with the sequence of unmatched item classes
	present in the buffer.
	More specifically, we consider the sequence
	$c = (c_1, \ldots, c_n)$,
	where $n$ is the number of items in the buffer
	and $c_p$ is the class of the $p$-th oldest item,
	for each $p \in \{1, \ldots, n\}$.
	In particular, $c_1$ is the class of the item
	that has been in the buffer the longest.
	\figurename~\ref{fig:toy-state}
	shows a possible buffer state
	associated with the compatibility graph
	of \figurename~\ref{fig:toy-graph},
	with the newest items on the left.
	The corresponding state descriptor
	is $c = (4, 4, 1, 4, 1)$.
	We assume that the state is initially empty,
	with $n = 0$, in which case we write $c = \varnothing$.
	The assumption that each class of items
	arrives according to an independent Poisson process
	guarantees that the stochastic process
	defined by the evolution
	of this state descriptor over time
	has the Markov property.
	When the system is in state $c = (c_1, \ldots, c_n)$
	and a class-$i$ item arrives,
	the new system state is
	$(c_1, \ldots, c_n, i)$
	if $c_p \nlink i$
	for each $p \in \{1, \ldots, n\}$
	and $(c_1, \ldots, c_{q-1}, c_{q+1}, \ldots, c_n)$,
	where
	$q = \min\{p \in \{1, \ldots, n\}: c_p \link i\}$,
	otherwise.
	The overall departure rate from each state
	is $\sum_{i \in \V} \alpha_i = 1$.
	
	\paragraph{State space}
	
	The system state $c = (c_1, \ldots, c_n)$
	belongs to the set $\V^*$
	made of the finite sequences over the alphabet $\V$,
	but the matching policy imposes
	additional restrictions on the form of this sequence.
	More specifically, the buffer
	cannot contain customers of classes $i, j \in \V$
	such that $i \link j$,
	as the newest of the two items would necessarily
	have been matched with the oldest
	(or with an older compatible item)
	upon arrival.
	In \figurename~\ref{fig:toy-state} for instance,
	an incoming class-$2$ item would
	be matched with the oldest class-$1$ item,
	so that items of classes~$1$ and $2$
	cannot be in the buffer at the same time.
	Therefore, in general, the state space
	of the Markov process
	defined by the system state
	is a strict subset of $\V^*$
	which we denote by $\C$.
	To better describe the structure of this state space,
	additional notation is introduced in the next paragraph.
	
	A nonempty set $\I \subseteq \V$
	is called an \emph{independent set}
	(of the compatibility graph)
	if $i \nlink j$ for each $i, j \in \I$.
	In other words,
	the set $\I$ is independent
	if the intersection of this set
	with the set $\E(\I)$
	of its neighbors is empty.
	The set of all independent sets
	is denoted by~$\ind$.
	As an example,
	the set of independent sets
	of the compatibility graph of \figurename~\ref{fig:toy-graph} is
	$\ind = \{\{1\}, \{2\}, \{3\}, \{4\}, \{1, 4\}, \{2, 4\}\}$.
	In general, if a set $\I \subseteq \V$
	is an independent set,
	then any nonempty subset of $\I$
	is also an independent set.
	For convenience,
	we also let $\inde = \ind \cup \{\emptyset\}$
	denote the set of independent sets
	completed with the empty set.
	
	For each $\I \in \inde$, we let
	$\C_\I$ denote the set of words in $\V^*$
	that contain exactly the classes in~$\I$.
	Equivalently, if $\V(c) \subseteq \V$
	denotes the set of item classes
	present in the word $c \in \V^*$,
	we have $c \in \C_\I$
	if and only if $\V(c) = \I$.
	Observe that $\C_\I$ is a strict subset
	of $\I^*$ in general,
	as words in $\I^*$
	need not contain \emph{all} classes in $\I$.
	In the special case where $\I = \emptyset$,
	the set $\C_\I$ contains only
	the empty state $\varnothing$.
	With this notation,
	the state space of the Markov process
	defined in the previous paragraph
	can be partitioned as follows:
	\begin{align} \label{eq:C-partition}
		\C = \bigcup_{\I \in \inde} \C_\I.
	\end{align}
	In particular, the set of item classes
	present in a feasible state
	of the matching model
	is an independent set.
	
	\begin{remark*}
		The stochastic non-bipartite matching model
		of \cite{MM16,MBM21,CDFB20,BMMR20}
		works in the same way
		as the one presented above,
		except that time is slotted and
		there is exactly one arrival during each time slot.
		More specifically,
		the classes of incoming items
		during different time slots are independent
		and identically distributed and,
		for each $i \in \V$,
		$\alpha_i$ is the probability
		that an incoming item is of class~$i$.
		One can verify
		that the (discrete-time) Markov chain associated with
		the state of this
		discrete-time stochastic matching model
		is the jump chain of
		the Markov process defined above.
		Since the departure rate
		from each state in the Markov process
		is $\sum_{i \in \V} \alpha_i = 1$,
		inspection of the balance equations reveals
		that the Markov chain
		and the Markov process
		have the same stationary measures.
		Consequently, the results in
		\cite[Theorem~1]{MBM21},
		derived for the discrete-time stochastic matching model,
		also apply to the continuous-time stochastic matching model
		defined above.
		These results will be recalled
		in Sections~\ref{subsec:steady}
		and~\ref{subsec:stability}.
		For simplicity,
		the continuous-time stochastic non-bipartite matching model
		introduced in Section~\ref{subsec:model}
		will be called the \emph{stochastic matching model}
		in the remainder.
	\end{remark*}

	\section{Product-form queueing model}
	\label{sec:oi}
	
	In Section~\ref{subsec:oi},
	we prove that
	the stochastic matching model
	of Section~\ref{subsec:model}
	is an instance of a queueing model,
	called an \emph{order-independent
		loss queue}~\cite{BK96},
	that is known to have a
	product-form stationary distribution.
	We then identify several extensions
	of the stochastic matching model
	that can also be cast
	as order-independent loss queues.
	In Sections~\ref{subsec:steady}
	and~\ref{subsec:stability},
	we use this observation
	to provide simpler proofs for two results
	from the literature on stochastic matching models.
	The reader who is more interested
	in understanding the long-term performance
	of the stochastic matching model
	can skip Section~\ref{subsec:oi} and move directly to Section~\ref{subsec:steady}.
	
	\subsection{Order-independent loss queue} \label{subsec:oi}
	
	Our definition of
	order-independent loss queues,
	given in Definition~\ref{def:oi} below,
	differs in spirit
	from that of \cite[Section~1]{BK96}
	in that conditions are given
	directly in terms of the associated
	Markov process,
	rather than in terms of the queueing system itself.
	In particular, although
	the name \emph{loss} usually entails rejecting customers,
	here this manifests itself via the fact that
	the state space $\C$ may be a strict subset of $\V^*$.
	This definition happens
	to be more convenient
	to show the connection
	with the stochastic matching model
	of Section~\ref{subsec:model},
	and preserves all the properties
	of the order-independent loss queues
	defined in \cite{BK96}.
	We also make a few restrictions
	(such as imposing constant arrival rates)
	compared to \cite{BK96}
	to avoid introducing superfluous notation.
	
	\begin{definition} \label{def:oi}
		Consider a queueing system
		with a set $\V = \{1, 2, \ldots, N\}$
		of customer classes.
		This queueing system
		is called
		an \emph{order-independent loss queue}
		with state space $\C \subseteq \V^*$,
		arrival-rate vector
		$\lambda = (\lambda_1, \lambda_2, \ldots, \lambda_N)$,
		and departure-rate function $\mu: \C \to \R_+$
		if the following conditions are satisfied:
		\begin{enumerate}[label=(\alph*)]
			\item The evolution over time
			of the sequence
			$c = (c_1, \ldots, c_n) \in \V$
			of classes of present customers,
			ordered by arrival,
			defines a Markov process
			such that:
			\begin{enumerate}[label=(\roman*)]
				\item The state space of this
				Markov process is $\C$.
				\item The transitions from each state
				$(c_1, \ldots, c_n) \in \C$
				are given as follows:
				\begin{itemize}
					\item For each class $i \in \V$
					such that $(c_1, \ldots, c_n, i) \in \C$,
					there is a transition
					from state~$c$
					to state $(c_1, \ldots, c_n, i)$
					with rate $\lambda_i$.
					\item For each position $p \in \{1, \ldots, n\}$
					such that $\mu(c_1, \ldots, c_p)
					> \mu(c_1, \ldots, c_{p-1})$,
					there is a transition
					from state~$c$ to state
					$(c_1, \ldots, c_{p-1}, c_{p+1}, \ldots, c_n)$
					with rate
					$\mu(c_1, \ldots, c_p)
					- \mu(c_1, \ldots, c_{p-1})$
					(with the convention that
					$(c_1, \ldots, c_{p-1}) = \varnothing$
					if $p = 1$).
				\end{itemize}
			\end{enumerate}
			\item The state space $\C$
			satisfies the following conditions:
			\begin{enumerate}[label=(\roman*)]
				\item if $(c_1, \ldots, c_n) \in \C$,
				then $(c_{\sigma(1)}, \ldots, c_{\sigma(n)})
				\in \C$
				for each permutation $\sigma$
				of $\{1, \ldots, n\}$, and
				\item if $(c_1, \ldots, c_n) \in \C \setminus \{\varnothing\}$,
				then $(c_1, \ldots, c_{n-1}) \in \C$.
			\end{enumerate}
			\item The departure-rate function $\mu$
			satisfies the following conditions:
			\begin{enumerate}[label=(\roman*)]
				\item $\mu(c_1, \ldots, c_n)
				= \mu(c_{\sigma(1)}, \ldots, c_{\sigma(n)})$
				for each state $(c_1, \ldots, c_n) \in \C$
				and permutation $\sigma$ on $\{1, \ldots, n\}$,
				\item $\mu(c_1, \ldots, c_n, i)
				\ge \mu(c_1, \ldots, c_n)$
				for each state $(c_1, \ldots, c_n) \in \C$
				and class $i \in \V$
				such that $(c_1, \ldots, c_n, i) \in \C$,
				\item $\mu(\varnothing) = 0$
				and $\mu(c_1, \ldots, c_n) > 0$
				for each $(c_1, \ldots, c_n)
				\in \C \setminus \varnothing$.
			\end{enumerate}
		\end{enumerate}
	\end{definition}
	
	Proposition~\ref{prop:oi} below
	is our first main contribution.
	It proves that
	the stochastic matching model
	is an order-independent loss queue.
	Note the form of the state space~$\C$
	defined by~\eqref{eq:C-partition},
	which differs from classic applications
	of \emph{loss} queueing models
	in that the number of items of each class
	can be arbitrarily large
	when it is not zero.
	Once Proposition~\ref{prop:oi} is verified,
	applying \cite[Theorem~1]{BK96}
	suffices to prove that
	the stationary measures
	of the Markov process
	defined in Section~\ref{subsec:ctmc}
	have the product form derived
	in~\cite[Theorem~1]{MBM21}.
	Both the result and its proof
	will be recalled
	in Section~\ref{subsec:steady}
	for completeness.
	
	\begin{proposition} \label{prop:oi}
		The stochastic matching model
		defined in Section~\ref{subsec:model}
		is an order-independent loss queue,
		in the sense of Definition~\ref{def:oi},
		with the state space $\C$
		defined by~\eqref{eq:C-partition},
		the arrival-rate vector
		$\lambda = \alpha = (\alpha_1, \alpha_2, \ldots, \alpha_N)$,
		and the departure-rate function
		$\mu: \C \to \R_+$ defined by
		\begin{align*}
			\mu(c)
			= \alpha(\E(\V(c)))
			\quad c \in \C.
		\end{align*}
	\end{proposition}
	
	\begin{proof}
		The details of checking
		the conditions of Definition~\ref{def:oi}
		are left to the reader.
		We simply make two remarks
		regarding the transitions
		of the Markov process
		defined in Section~\ref{subsec:ctmc}.
		First, for each state
		$c = (c_1, \ldots, c_n) \in \C$
		and class $i \in \V$,
		we have $(c_1, \ldots, c_n, i) \in \C$
		if and only if $i \notin \E(\V(c))$,
		that is, if and only if
		class~$i$ cannot be matched
		with any class present in state~$c$.
		Second, if the function $\mu: \C \to \R_+$
		is as defined in the proposition,
		then for each state $(c_1, \ldots, c_n) \in \C$
		and position $p \in \{1, \ldots, n\}$,
		the difference
		\begin{align*}
			\mu(c_1, \ldots, c_p)
			- \mu(c_1, \ldots, c_{p-1})
			= \alpha\left(
			\E_{c_p} \setminus \E(\V(c_1, \ldots, c_{p-1}))
			\right)
		\end{align*}
		gives the arrival rate of the classes
		that can be matched with class~$c_p$
		but not with the classes
		present in $(c_1, \ldots, c_{p-1})$.
	\end{proof}
	
	\begin{remark*}
		The versatility of the order-independent framework
		allows us to generalize
		the stochastic matching model
		introduced in Section~\ref{subsec:model}
		in several directions
		while preserving its product-form nature.
		For instance,
		we again obtain
		an order-independent loss queue
		if items of a given class can
		be matched with one another.
		The main difference with the original model
		is that the state space is reduced,
		as the buffer can
		contain at most one customer
		of such a class at a time.
		We note that this extension was considered
		in the recent work~\cite{BMMR20},
		and that the proofs of Theorems~\ref{theo:pic}
		and \ref{theo:stability}
		in Sections~\ref{subsec:steady}
		and~\ref{subsec:stability}
		can be adapted to provide a simpler proof of
		\cite[Theorem~1]{BMMR20}.
		Another extension,
		which again leads to
		an order-independent loss queue,
		consists of
		describing compatibilities with a directed graph
		(instead of an undirected graph),
		so that there is an edge
		from class~$i$ to class~$j$
		if incoming class-$i$ items can be matched
		with class-$j$ items
		(which does not imply that incoming class-$j$ items
		can be matched with class-$i$ items).
		Yet another extension consists of
		allowing for item abandonment~\cite{TFT13},
		meaning that an item may leave
		the system without being matched.
		The stability region
		and product-form stationary distribution
		of these extensions
		can be characterized by following a similar approach
		as in the proofs of Theorems~\ref{theo:pic}
		and~\ref{theo:stability},
		or by applying known results
		on order-independent queues
		and quasi-reversible queues,
		a more general class of queueing models
		studied in \cite[Chapter~3]{kelly}
		and~\cite{C11}.
		A last extension consists
		of imposing an upper bound on the number
		of unmatched items (either of each class or in total),
		or any other rejection rule
		that again leads to an order-independent loss queue.
		In this case, the stationary distribution
		is the restriction of that recalled in Theorem~\ref{theo:pic}
		to the set of reachable states
		(after renormalization),
		and no stability condition is needed.
	\end{remark*}
	
	\begin{remark*}
		The usefulness of the equivalence
		between stochastic matching models
		and conventional queueing models
		is particularly evident
		under the first-come-first-matched policy,
		as the corresponding queueing model is known
		to have a product-form stationary distribution.
		However, this equivalence may also prove useful
		to transpose existing scheduling policies
		into matching policies,
		as was done in~\cite{BGM13,MM16,BMMR20},
		or to adapt exact or approximate analysis methods.
	\end{remark*}
	
	\subsection{Product-form stationary distribution} \label{subsec:steady}
	
	The Markov process
	defined in Section~\ref{subsec:ctmc}
	is irreducible.
	Indeed, for each state
	$c = (c_1, \ldots, c_n) \in \C$
	and each state $d = (d_1, \ldots, d_m) \in \C$,
	we can first go
	from state $c$ to state $\varnothing$
	via $n$ transitions
	corresponding to arrivals of items
	compatible with the items in state $c$,
	and then from state~$\varnothing$ to state~$d$
	via $m$ transitions
	corresponding to arrivals of the items
	present in state~$d$.
	Note that the first $n$ transitions
	occur with a positive probability
	because the compatibility graph
	contains no isolated node;
	additionally, the latter $m$ transitions
	indeed lead to state~$d$
	because the set of classes in $d$
	is an independent set by definition of $\C$.
	
	Theorem~\ref{theo:pic} below
	recalls the form of the stationary measures
	of the Markov process
	defined in Section~\ref{subsec:ctmc}.
	As observed before,
	this theorem is analogous
	to \cite[Theorem~1]{MBM21},
	which was stated for the discrete-time
	variant of the stochastic matching model.
	Thanks to the relation with
	order-independent loss queues
	proved
	in Proposition~\ref{prop:oi},
	we propose a simpler proof
	that consists of explicitly writing
	the (partial) balance equations
	of the Markov process.
	This proof, which is a special case
	of that of \cite[Theorem~1]{BK96},
	is sketched below for completeness.
	
	\begin{theorem} \label{theo:pic}
		The stationary measures of
		the Markov process
		associated with the system state
		are of the form
		\begin{align} \label{eq:pic-def}
			\pi(c_1, \ldots, c_n)
			= \pi(\varnothing)
			\prod_{p = 1}^n
			\frac{\alpha_{c_p}}
			{\alpha(\E(\V(c_1, \ldots, c_p)))},
			\quad (c_1, \ldots, c_n) \in \C,
		\end{align}
		where $\pi(\varnothing)$
		is a normalization constant.
		The system is stable, in the sense that
		this Markov process is ergodic, if and only if
		\begin{align} \label{eq:stability-complicated}
			\sum_{(c_1, \ldots, c_n) \in \C}
			\prod_{p = 1}^n
			\frac{\alpha_{c_p}}
			{\alpha(\E(\V(c_1, \ldots, c_p)))}
			< \infty,
		\end{align}
		in which case the stationary distribution
		of the Markov process is given by~\eqref{eq:pic-def}
		with the normalization constant
		\begin{align} \label{eq:normalization}
			\pi(\varnothing)
			= \left(
			\sum_{(c_1, \ldots, c_n) \in \C}
			\prod_{p = 1}^n
			\frac{\alpha_{c_p}}
			{\alpha(\E(\V(c_1, \ldots, c_p)))}
			\right)^{-1}.
		\end{align}
	\end{theorem}
	
	\begin{proof}[Sketch of proof]
		We will prove that
		any measure $\pi$ of the form~\eqref{eq:pic-def}
		satisfies the following
		partial balance equations
		in each state~$c = (c_1, \ldots, c_n) \in \C$:
		\begin{itemize}
			\item Equalize the flow out of state~$c$
			due to the arrival of an item
			that can be matched with a present item
			and the flow into state~$c$
			due to the arrival of an item
			that cannot be matched with a present item
			(if $c \neq \varnothing$):
			\begin{align} \label{eq:balance1}
				\pi(c) \alpha(\E(\V(c)))
				= \pi(c_1, \ldots, c_{n-1}) \alpha_{c_n}.
			\end{align}
			\item Equalize the flow out of state~$c$
			due to the arrival of an item
			of a class $i \in \V \setminus \E(\V(c))$
			that cannot be matched with a present item
			and the flow into state~$c$
			due to the departure of an item of this class:
			\begin{align} \label{eq:balance2}
				\pi(c) \alpha_i
				= \sum_{p = 1}^{n+1}
				\pi(c_1, \ldots, c_{p-1}, i, c_p, \ldots, c_n)
				\alpha(
				\E_i \setminus
				\E(\V(c_1, \ldots, c_{p-1}))
				),
			\end{align}
			with the convention that
			$(c_1, \ldots, c_{p-1}) = \varnothing$
			if $p = 1$
			and $(c_p, \ldots, c_n) = \varnothing$
			if $p = n + 1$.
			Observe that,
			for each $i \in \V \setminus \E(\V(c))$
			and $p \in \{1, \ldots, n+1\}$,
			state
			$(c_1, \ldots, c_{p-1}, i, c_p, \ldots, c_n)$
			belongs to $\C$ because
			the set of item classes present in this state,
			$\V(c) \cup \{i\}$,
			is an independent set.
		\end{itemize}
		Proving that a measure $\pi$
		satisfies the partial balance
		equations~\eqref{eq:balance1}
		and \eqref{eq:balance2}
		is sufficient to prove that
		$\pi$ is a stationary measure of
		the Markov process defined by the system state,
		as summing these partial balance equations
		yields the balance equations
		of the Markov process.
		We can easily verify that
		any measure $\pi$ of the form~\eqref{eq:pic-def}
		satisfies~\eqref{eq:balance1}.
		That this measure also satisfies~\eqref{eq:balance2}
		can be proved by induction
		over the buffer length~$n$,
		in a similar way as in \cite[Theorem~1]{BKK95}.
		
		The stability condition~\eqref{eq:stability-complicated}
		means that there exists a stationary measure
		whose sum is finite,
		which is indeed necessary and sufficient
		for ergodicity of the Markov process.
		Lastly, the constant $\pi(\varnothing)$
		given by~\eqref{eq:normalization}
		is chosen such that the stationary distribution sums to one.
	\end{proof}
	
	Multiplying all arrival rates
	$\alpha_i$ for $i \in \V$
	by the same positive constant
	does not modify the form of
	the stationary measures~\eqref{eq:pic-def}.
	Indeed, this amounts to rescaling time
	without modifying the relative importance
	of the states.
	In particular, as already stated
	in Section~\ref{subsec:model}, we can assume
	without loss of generality
	that $\sum_{i \in \V} \alpha_i = 1$,
	so that $\alpha_i$ is also
	the fraction of arrivals
	that are of class~$i$.
	In the remainder,
	it will also be more convenient to define
	the stationary measures~\eqref{eq:pic-def}
	with the following recursion:
	\begin{align} \label{eq:pic}
		\pi(c_1, \ldots, c_n)
		= \frac{\alpha_{c_n}}
		{\alpha(\E(\V(c_1, \ldots, c_n)))}
		\pi(c_1, \ldots, c_{n-1}),
		\quad (c_1, \ldots, c_n)
		\in \C \setminus \{\varnothing\}.
	\end{align}
	
	\subsection{Stability condition} \label{subsec:stability}
	
	We now recall the stability condition
	derived in~\cite[Theorem 1]{MBM21}
	and again provide a simpler proof based on
	the relation with order-independent loss queues.
	
	\begin{theorem} \label{theo:stability}
		The stochastic matching model
		is stable,
		in the sense that the
		associated Markov process is ergodic,
		if and only if
		\begin{align} \label{eq:stability}
			\alpha(\I) < \alpha(\E(\I)),
			\quad \I \in \ind.
		\end{align}
	\end{theorem}
	
	\begin{proof}[Sketch of proof]
		The proof consists of showing
		that~\eqref{eq:stability-complicated}
		is satisfied if and only if
		the arrival rates satisfy~\eqref{eq:stability}.
		The key idea is
		to observe that the denominators
		in~\eqref{eq:stability-complicated}
		only depend on the set of classes
		in the sequence $(c_1, \ldots, c_p)$.
		Guided by this observation,
		we can define a new partition
		of the state space~$\C$,
		that is a refinement of
		the partition $(\C_\I, \I \in \inde)$,
		such that the restriction of the sum
		in~\eqref{eq:stability-complicated}
		to each part of this new partition
		can be written
		as a finite product of geometric series
		that are convergent
		if and only if
		the conditions~\eqref{eq:stability}
		are satisfied. The details of the proof follow
		along the same lines
		as in the proof of Theorem~1
		in the appendix of~\cite{BC17}.
		The only difference is that
		the set of customer classes
		that can be in the queue
		at the same time
		is an independent set,
		so that the stability condition
		is also restricted
		to independent sets.
	\end{proof}
	
	The stability
	condition~\eqref{eq:stability} means that,
	for each independent set $\I$,
	there are enough items
	that can be matched with the classes in $\I$
	to prevent items of these classes
	from building up in the buffer.
	As proved in \cite{MM16},
	no matching policy can make the system stable
	if the arrival rates do not satisfy
	the stability condition~\eqref{eq:stability}
	(in particular, first-come-first-matched is
	maximally stable).
	It is also shown in \cite{MM16} that
	we would not be able to find arrival rates
	that satisfy this stability condition
	if the compatibility graph were bipartite.
	To better understand why
	the presence of odd cycles is necessary,
	we examine two toy examples.
	First consider a stochastic matching model
	similar to that of Section~\ref{subsec:model},
	except that the graph is bipartite
	and consists of only two classes connected by an edge.
	The stability condition~\eqref{eq:stability}
	becomes
	$\alpha_1 < \alpha_2$ and $\alpha_2 < \alpha_1$,
	and can therefore not be satisfied.
	Consistently, if we examine
	the associated Markov process,
	we obtain a two-sided infinite
	birth-and-death process,
	with transition rate $\alpha_1$ in one direction
	and $\alpha_2$ in the other,
	that is either transient
	(if $\alpha_1 \neq \alpha_2$)
	or null recurrent
	(if $\alpha_1 = \alpha_2$),
	but cannot be positive recurrent.
	Now, if we instead consider a stochastic matching model
	with three classes that are all compatible,
	the stability condition~\eqref{eq:stability} becomes
	$\alpha_1 < \alpha_2 + \alpha_3$, $\alpha_2 < \alpha_1 + \alpha_3$,
	and $\alpha_3 < \alpha_1 + \alpha_2$,
	and is satisfied for instance when
	$\alpha_1 = \alpha_2 = \alpha_3 = \frac13$.
	If we examine the associated Markov process,
	we obtain a three-sided infinite birth-and-death process
	that is positive recurrent under the above conditions.
	In the remainder, we will assume that
	the stability condition~\eqref{eq:stability}
	is satisfied
	and let $\pi$ denote the stationary distribution
	of the Markov process.

	\section{Performance metrics} \label{sec:perf}
	
	Another consequence of Proposition~\ref{prop:oi}
	is that we can adapt numerical methods
	originally developed for order-independent
	queues~\cite{SV15,C19,CBL21}
	to stochastic matching models.
	Note that calculating
	average performance metrics
	associated with \emph{loss} product-form
	queueing models
	is considered
	a hard problem in the queueing literature~\cite{LMK94}.
	Our task here is simplified
	by the fact that
	the number of items of a given class
	can be arbitrarily large if it is not zero,
	so that we can apply methods
	developed for \emph{open} queues.
	Also, since we established that
	the continuous-time stochastic matching model
	and its discrete-time counterpart
	have the same stationary distribution,
	the formulas derived in this section
	can be applied without change
	to the discrete-time
	variant of the model.
	
	\subsection{Probability of an empty system} \label{subsec:perf-empty}
	
	We first propose a new formula
	to compute the normalization constant
	of the stationary distribution
	of the stochastic matching model.
	With a slight abuse of notation,
	for each $\I \in \inde$,
	we let $\pi(\I)$ denote the stationary probability
	that the set of unmatched item classes is~$\I$, that is,
	\begin{align} \label{eq:piI-def}
		\pi(\I) = \sum_{c \in \C_\I} \pi(c),
		\quad \I \in \inde.
	\end{align}
	Equation~\eqref{eq:C-partition} guarantees that
	$\sum_{\I \in \inde} \pi(\I) = 1$.
	The following proposition
	gives a recursive formula
	for $\pi(\I)$ for each $\I \in \ind$.
	We will see later that
	this recursive formula
	also gives a closed-form expression for
	the normalization constant
	$\pi(\emptyset) = \pi(\varnothing)$.
	The proof technique is
	similar to that of \cite[Theorem~5.1]{C19},
	which was itself inspired from \cite[Theorem~4]{SV15},
	but the proof is more straightforward
	because we skip the intermediary state aggregation
	considered in these two works.
	
	\begin{proposition} \label{prop:piI-rec}
		The stationary distribution of
		the set of unmatched item classes
		satisfies the recursion
		\begin{align} \label{eq:piI-rec}
			\pi(\I) &=
			\frac{\sum_{i \in \I}
				\alpha_i \pi(\I \setminus \{i\})}
			{\alpha(\E(\I)) - \alpha(\I)},
			\quad \I \in \ind.
		\end{align}
	\end{proposition}
	
	\begin{proof}
		Let $\I \in \ind$.
		The proof of~\eqref{eq:piI-rec} relies
		on the following partition of
		the set $\C_\I$:
		\begin{align} \label{eq:recursion}
			\C_\I
			= \bigcup_{i \in \I}
			\left(
			(\C_\I \cdot i)
			\cup (\C_{\I \setminus \{i\}} \cdot i)
			\right),
		\end{align}
		where the unions are disjoint and,
		for each $\D \subseteq \C$
		and $i \in \V$, we write
		$\D \cdot i = \{(c_1, \ldots, c_n, i):
		(c_1, \ldots, c_n) \in \D\}$.
		We first apply \eqref{eq:pic}
		and \eqref{eq:piI-def}
		to rewrite $\pi(\I)$ as follows:
		\begin{align*}
			\pi(\I)
			= \frac1{\alpha(\E(\I))}
			\sum_{c \in \C_\I}
			\alpha_{c_n}
			\pi(c_1, \ldots, c_{n-1}).
		\end{align*}
		Then applying~\eqref{eq:recursion} yields
		\begin{align*}
			\alpha(\E(\I)) \pi(\I)
			=
			\sum_{i \in \I}
			\alpha_i
			\sum_{c \in \C_\I} \pi(c)
			+ \sum_{i \in \I}
			\alpha_i
			\sum_{c \in \C_{\I \setminus \{i\}}}
			\pi(c)
			= \sum_{i \in \I} \alpha_i \pi(\I)
			+ \sum_{i \in \I} \alpha_i
			\pi(\I \setminus \{i\}).
		\end{align*}
		Rearranging the terms yields~\eqref{eq:piI-rec}.
	\end{proof}
	
	Proposition~\ref{prop:piI-rec}
	gives a closed-form expression for
	the normalization constant
	$\pi(\emptyset) = \pi(\varnothing)$.
	It suffices to observe that
	the unnormalized stationary measure
	$\psi(\I) = \pi(\I) / \pi(\emptyset)$
	also satisfies recursion~\eqref{eq:piI-rec}.
	Therefore, we can first compute
	$\psi(\I)$ for each $\I \in \ind$
	by applying recursion~\eqref{eq:piI-rec}
	with the base case $\psi(\emptyset) = 1$,
	and then derive $\pi(\emptyset)$ by normalization,
	that is, $\pi(\emptyset)
	= (\sum_{\I \in \inde} \psi(\I))^{-1}$.
	Since there is a finite number of independent sets,
	this provides us with a closed-form expression
	for the normalization constant.
	Note that the denominator in~\eqref{eq:piI-rec}
	is always positive thanks to
	the stability condition~\eqref{eq:stability}.
	
	One can verify that unfolding the above recursion
	yields the closed-form expression
	derived in \cite[Equation~(5)]{MBM21}.
	The advantage of the above recursion
	is that it avoids duplicate calculations.
	More specifically, the overall complexity
	to evaluate the normalization constant
	using~\eqref{eq:piI-rec}
	with dynamic programming is $O(MN)$,
	where $M = |\ind|$ is the number of independent sets
	and $N = |\V|$ is the number of item classes.
	By comparison,
	if we naively apply \cite[Equation~(5)]{MBM21},
	the complexity is $O(MNP \times P!)$,
	where $P$ is the independence number
	of the compatibility graph,
	that is, the cardinality
	of a maximum independent set in this graph.
	In both cases however, we need to enumerate
	all independent sets of the compatibility graph,
	which is an NP-hard optimization problem in general.
	Consequently, both formulas
	can be applied to make calculations only
	when there are few classes,
	when symmetries between classes
	can be exploited to reduce complexity
	(see for instance~\cite{BCSV17}),
	or when the graph structure guarantees that
	the number of independent sets is small.
	The same remark holds for the formulas derived in
	Propositions~\ref{prop:waiting}
	to \ref{prop:lst-rec} below.
	However, we will see in Section~\ref{sec:app} that,
	even when these conditions are not meet,
	these formulas can be used to
	gain insight into the impact
	of parameters on performance.
	
	\subsection{Waiting probability} \label{subsec:perf-waiting-proba}
	
	The result of Proposition~\ref{prop:piI-rec}
	gives us directly a closed-form expression
	for the waiting probability
	of the items of each class.
	Indeed, it follows from the PASTA property that,
	for each $i \in \V$, the waiting probability
	of class-$i$ items is given by
	\begin{align} \label{eq:waiting}
		\omega_i
		= \sum_{\substack{
				\I \in \inde: \\
				i \notin \E(\I)
		}} \pi(\I).
	\end{align}
	Upon recalling that $\sum_{i \in \V} \alpha_i = 1$,
	Proposition~\ref{prop:waiting} below
	shows that the average waiting probability
	is simply given by $\frac12$.
	This result is not surprising.
	Indeed, each matching occurs between two items,
	one of that has just arrived
	(and therefore does not wait)
	and another that was waiting.
	However, this result highlights that
	improving the performance of one class
	(in terms of the waiting probability)
	necessarily comes at the expense of another class.
	
	\begin{proposition} \label{prop:waiting}
		The average waiting probability is given by
		$\sum_{i \in \V} \alpha_i \omega_i = \frac12$.
	\end{proposition}
	
	\begin{proof}
		We will verify that the per-class
		waiting probabilities satisfy
		the following conservation equation:
		\begin{align} \label{eq:balance}
			\sum_{i \in \V} \alpha_i \omega_i
			= \sum_{i \in \V} \alpha_i (1 - \omega_i),
		\end{align}
		after which the result follows
		by rearranging the terms.
		We have successively:
		\begin{align*}
			\sum_{i \in \V} \alpha_i (1 - \omega_i)
			&= \sum_{i \in \V} \alpha_i
			\sum_{\substack{
					\I \in \ind: \\
					i \in \E(\I)
			}} \pi(\I)
			= \sum_{\I \in \ind} \pi(\I)
			\sum_{i \in \E(\I)} \alpha_i
			= \sum_{\I \in \ind} \pi(\I) \alpha(\E(\I)),
		\end{align*}
		where the first equality follows from
		the definition of $\omega_i$,
		the second by exchanging the two sums,
		and the third from the definition of $\alpha(\E(\I))$.
		Applying~\eqref{eq:piI-rec}
		to the last member of this equality yields
		\begin{align*}
			\sum_{i \in \V} \alpha_i (1 - \omega_i)
			&= \sum_{\I \in \ind} \left(
			\sum_{i \in \I} \alpha_i \pi(\I)
			+ \sum_{i \in \I} \alpha_i \pi(\I \setminus \{i\})
			\right)
			= \sum_{i \in \V} \alpha_i
			\sum_{\substack{
					\I \in \ind: \\
					i \in \I
			}} \left(
			\pi(\I) + \pi(\I \setminus \{i\})
			\right),
		\end{align*}
		where the second equality follows
		by exchanging the sums.
		Equation~\eqref{eq:balance} follows
		by observing that, for each $i \in \V$, we have
		\begin{align*}
			\sum_{\substack{
					\I \in \ind: \\
					i \in \I
			}} \left(
			\pi(\I) + \pi(\I \setminus \{i\})
			\right)
			= \sum_{\substack{
					\I \in \inde: \\
					i \notin \E(\I)
			}} \pi(\I)
			= \omega_i,
		\end{align*}
		where the first equality follows
		from a change of variable
		and the second by definition of $\omega_i$.
	\end{proof}

	\subsection{Mean number of unmatched items
		and mean matching time}  \label{subsec:perf-mean-number}
	
	Consider a random vector $X = (X_1, X_2, \ldots, X_N)$
	distributed as
	the vector of numbers of unmatched items of each class
	in the stationary regime
	in the stochastic matching model
	of Section~\ref{subsec:model}.
	For now, we focus on the means
	$L_i = \esp{X_i}$ for each $i \in \V$,
	and we also let $L = \sum_{i \in \V} L_i$
	be the mean number of unmatched items,
	all classes included.
	Propositions~\ref{prop:LiI-rec}
	and~\ref{prop:LI-rec} below
	give a closed-form expression
	for these quantities
	by using a similar aggregation technique
	as in the proof of Proposition~\ref{prop:piI-rec}.
	As observed in~\cite[Section~3.4]{MBM21},
	these can be used to calculate
	the mean matching time of items,
	that is, the mean time between
	the arrival of an item
	and its matching with another item.
	Indeed, according to Little's law,
	the mean matching time of class-$i$ items
	is $L_i / \alpha_i$, for each $i \in \V$,
	and the overall mean matching time
	is given by $L / (\sum_{i \in \V} \alpha_i) = L$
	(since we assumed that $\sum_{i \in \V} \alpha_i = 1$).
	
	\begin{proposition}	\label{prop:LiI-rec}
		For each $i \in \V$, the mean
		number of unmatched class-$i$ items is given by
		$$
		L_i = \sum_{\I \in \ind: i \in \I}
		L_i(\I) \pi(\I),
		$$
		where $L_i(\I)$ is the mean number
		of unmatched class-$i$ items
		given that the set of unmatched items is $\I$, and is
		given by the recursion
		\begin{align} \label{eq:LiI-rec}
			L_i(\I) \pi(\I)
			= \frac{
				\alpha_i \pi(\I)
				+ \alpha_i \pi(\I \setminus \{i\})
				+ \sum_{j \in \I \setminus \{i\}}
				\alpha_j L_i(\I \setminus \{j\})
				\pi(\I \setminus \{j\})
			}{\alpha(\E(\I)) - \alpha(\I)},
			\quad \I \in \ind: i \in \I,
		\end{align}
		with the base case $L_i(\I) = 0$
		for each $\I \in \inde$ such that $i \notin \I$.
	\end{proposition}
	
	\begin{proof}
		The proof follows similar steps
		as that of Proposition~\ref{prop:piI-rec},
		with some technical complications.
		Let $i \in \V$.
		We have
		$L_i = \sum_{\I \in \inde}
		L_i(\I) \pi(\I)$ where,
		for each $\I \in \inde$,
		$L_i(\I)$ is the mean number
		of unmatched class-$i$ items
		given that the set of unmatched
		item classes is $\I$.
		We have directly that $L_i(\I) = 0$
		if $i \notin \I$.
		If $i \in \I$, we have
		\begin{align*}
			L_i(\I) \pi(\I)
			&= \sum_{c \in \C_\I} |c|_i \pi(c)
			= \sum_{c \in \C_\I} |c|_i
			\frac{\alpha_{c_n}}{\alpha(\E(\I))}
			\pi(c_1, \ldots, c_{n-1}),
		\end{align*}
		where $|c|_i$ is the number of
		unmatched class-$i$ items
		in state~$c$, for each $c \in \C$,
		and the second equality follows from \eqref{eq:pic}.
		Applying partition~\eqref{eq:recursion} yields
		\begin{align*}
			\alpha(\E(\I)) L_i(\I) \pi(\I)
			&= \begin{aligned}[t]
				&\alpha_i \left[
				\sum_{c \in \C_\I}
				\left(|c|_i + 1 \right) \pi(c)
				+ \sum_{c \in \C_{\I \setminus \{i\}}}
				\left(0 + 1 \right) \pi(c)
				\right] \\
				&+ \sum_{j \in \I \setminus \{i\}}
				\alpha_j \left[
				\sum_{c \in \C_\I}
				|c|_i \pi(c)
				+ \sum_{c \in \C_{\I \setminus \{j\}}}
				|c|_i \pi(c)
				\right].
			\end{aligned}
		\end{align*}
		This, in turn, can be rewritten as
		\begin{align*}
			\alpha(\E(\I)) L_i(\I) \pi(\I)
			&= \begin{aligned}[t]
				&\alpha_i \left[
				L_i(\I) \pi(\I) + \pi(\I)
				+ \pi(\I \setminus \{i\})
				\right]
				+ \sum_{j \in \I \setminus \{i\}} \alpha_j \left[
				L_i(\I) \pi(\I)
				+ L_i(\I \setminus \{j\})
				\pi(\I \setminus \{j\})
				\right],
			\end{aligned} \\
			&= \begin{aligned}[t]
				&\alpha(\I) L_i(\I) \pi(\I)
				+ \alpha_i \pi(\I)
				+ \alpha_i \pi(\I \setminus \{i\})
				+ \sum_{j \in \I \setminus \{i\}}
				L_i(\I \setminus \{j\}) \pi(\I \setminus \{j\}).
			\end{aligned}
		\end{align*}
		Rearranging the terms yields~\eqref{eq:LiI-rec}.
	\end{proof}
	
	To the best of our knowledge,
	no previous work
	obtained the above formula
	for the stochastic matching model
	(either in recursive or in non-recursive form).
	We now provide an analogous formula
	to directly calculate
	the mean number of unmatched items,
	all classes included,
	without needing to calculate
	the mean number of unmatched items of each class.
	
	\begin{proposition} \label{prop:LI-rec}
		The mean number
		of unmatched items is given by
		\begin{align*}
			L = \sum_{\I \in \ind} L(\I) \pi(\I),
		\end{align*}
		where $L(\I)$ is the mean number of unmatched items
		given that the set of unmatched items is $\I$, and is
		given by the recursion
		\begin{align} \label{eq:LI-rec}
			L(\I) \pi(\I)
			= \frac{
				\alpha(\E(\I)) \pi(\I)
				+ \sum_{i \in \I} \alpha_i
				L(\I \setminus \{i\})
				\pi(\I \setminus \{i\})
			}{\alpha(\E(\I)) - \alpha(\I)},
			\quad \I \in \ind,
		\end{align}
		with the base case $L(\emptyset) = 0$.
	\end{proposition}
	
	\begin{proof}
		The result follows by applying~\eqref{eq:LiI-rec}
		to $L(\I) = \sum_{i \in \I} L_i(\I)$
		and simplifying the expression with \eqref{eq:piI-rec}.	
	\end{proof}
	
	Unfolding recursion~\eqref{eq:LI-rec}
	yields the formula given in
	\cite[Proposition~4]{CDFB20}
	but, as for the normalization constant,
	applying dynamic programming
	leads to a significant reduction of complexity
	compared to a naive application of this proposition.
	More specifically,
	the time complexity
	to compute $L_i$ for some $i \in \V$
	or $L$ is $O(MN)$ if we use our approach,
	while the time complexity
	to compute $L$ is $O(MNP \times P!)$
	if we naively apply the formula
	of \cite[Proposition~4]{CDFB20}.

	\subsection{Distribution of the number of unmatched items}  \label{subsec:perf-distribution-number}
	
	As before, let $X = (X_1, X_2, \ldots, X_N)$
	be a random vector distributed like
	the vector of numbers of unmatched items of each class
	in stationary regime
	in the stochastic matching model
	of Section~\ref{subsec:model}.
	A closed-form expression for
	the joint probability distribution of the vector~$X$
	can be derived in the same way as in
	\cite[Theorem~3.2, Equation~(3.1)]{C19}.
	Here, we are instead interested in
	the probability generating function $g_X$
	of this vector, defined as
	\begin{align} \label{eq:pgf-def}
		g_X(z)
		&= \esp{\prod_{i \in \V} {z_i}^{X_i}},
		\quad z = (z_1, z_2, \ldots, z_N) \in \Z,
	\end{align}
	where $\Z = \{z = (z_1, z_2, \ldots, z_N) \in \R_+^N:
	z_1 \le 1, z_2 \le 1, \ldots, z_N \le 1\}$.
	One can show that this generating series
	converges whenever $z \in \Z$.
	The following notation will be useful
	to state Proposition~\ref{prop:pgf-rec} below.
	First, for each
	$\lambda = (\lambda_1, \lambda_2,
	\ldots, \lambda_N) \in \R_+^N$,
	we write $\lambda(\A) = \sum_{i \in \A} \lambda_i$
	for each $\A \subseteq \V$.
	Also, given two vectors
	$\lambda = (\lambda_1, \lambda_2,
	\ldots, \lambda_N) \in \R_+^N$
	and $\mu = (\mu_1, \mu_2, \ldots, \mu_N) \in \R_+^N$,
	we let $\lambda \mu
	= (\lambda_1 \mu_1, \lambda_2 \mu_2, \ldots, \lambda_N \mu_N)$
	denote their elementwise product.
	We are now in position to give
	a closed-form expression for
	the probability generating function of the vector $X$.
	Both this result and its proof are adapted from
	\cite[Proposition~3.1]{CBL21}.
	
	\begin{proposition} \label{prop:pgf-rec}
		The probability generating function of
		the numbers of unmatched items
		of each class in the buffer
		is given by
		\begin{align} \label{eq:pgf-rec}
			g_X(z)
			= \frac
			{\sum_{\I \in \inde} \psi_{\alpha z, \alpha}(\I)}
			{\sum_{\I \in \inde} \psi_{\alpha, \alpha}(\I)},
		\end{align}
		where, for each
		$\lambda = (\lambda_1, \lambda_2,
		\ldots, \lambda_N) \in \R_+^N$
		and $\mu = (\mu_1, \mu_2, \ldots, \mu_N) \in \R_+^N$
		such that
		$\lambda(\I) < \mu(\E(\I))$
		for each $\I \in \ind$,
		the set function $\psi_{\lambda, \mu}$
		is defined by the recursion
		\begin{align} \label{eq:pilambda-rec}
			\psi_{\lambda, \mu}(\I)
			= \frac
			{\sum_{i \in \I} \lambda_i
				\psi_{\lambda, \mu}(\I \setminus \{i\})}
			{\mu(\E(\I)) - \lambda(\I)},
			\quad \I \in \ind,
		\end{align}
		with the base case
		$\psi_{\lambda, \mu}(\emptyset) = 1$.
	\end{proposition}
	
	\begin{proof}
		It suffices to observe that,
		according to~\eqref{eq:pic-def},
		\eqref{eq:normalization},
		and \eqref{eq:pgf-def},
		$g_X(z)$ can be rewritten as
		\begin{align*}
			g_X(z)
			&= \frac{ \displaystyle
				\sum_{c \in \C} \prod_{p = 1}^n
				\frac{z_{c_p} \alpha_{c_p}}{\alpha(\E(\V(c_1, \ldots, c_p)))}
			}{ \displaystyle
				\sum_{c \in \C} \prod_{p = 1}^n
				\frac{\alpha_{c_p}}{\alpha(\E(\V(c_1, \ldots, c_p)))}
			}
			= \frac{ \displaystyle
				\sum_{\I \in \inde}
				\sum_{c \in \C_\I} \prod_{p = 1}^n
				\frac{z_{c_p} \alpha_{c_p}}{\alpha(\E(\V(c_1, \ldots, c_p)))}
			}{ \displaystyle
				\sum_{\I \in \inde}
				\sum_{c \in \C_\I} \prod_{p = 1}^n
				\frac{\alpha_{c_p}}{\alpha(\E(\V(c_1, \ldots, c_p)))}
			}
			= \frac{ \displaystyle
				\sum_{\I \in \inde} \psi_{z \alpha, \alpha}(\I)
			}{ \displaystyle
				\sum_{\I \in \inde} \psi_{\alpha, \alpha}(\I)
			},
		\end{align*}
		where, for each
		$\lambda = (\lambda_1, \lambda_2, \ldots,
		\lambda_N) \in \R_+^N$
		and $\mu = (\mu_1, \mu_2, \ldots, \mu_N) \in \R_+^N$
		such that $\lambda(\I) < \mu(\E(\I))$
		for each $\I \in \ind$, we have
		\begin{align*}
			\psi_{\lambda, \mu}(\I)
			= \sum_{c \in \C_\I} \prod_{p = 1}^n
			\frac{\lambda_{c_p}}{\mu(\E(\V(c_1, \ldots, c_p)))},
			\quad \I \in \inde.
		\end{align*}
		The recursive expression~\eqref{eq:pilambda-rec}
		follows from a similar argument
		as in the proof of Proposition~\ref{prop:piI-rec}.
	\end{proof}
	
	Taking $\lambda = \mu = \alpha$
	in~\eqref{eq:pilambda-rec}
	yields~\eqref{eq:piI-rec},
	so that in particular
	the denominator in \eqref{eq:pgf-rec}
	is the normalization constant
	of the stationary distribution~\eqref{eq:pic-def}.
	Referring back to the
	order-independent loss queue
	introduced in Section~\ref{sec:oi},
	the vector $\lambda = (\lambda_1, \lambda_2, \ldots, \lambda_N)$
	in~\eqref{eq:pilambda-rec}
	can be seen as the arrival rates
	of the customer classes,
	while the vector
	$\mu = (\mu_1, \mu_2, \ldots, \mu_N)$
	could be seen as service rates.
	
	\subsection{Distribution of the matching time}  \label{subsec:perf-distribution-time}
	
	Let $T = (T_1, T_2, \ldots, T_N)$ denote
	a vector distributed like the vector of
	matching times of each class
	in the stationary regime
	in the stochastic matching model
	of Section~\ref{subsec:model}.
	For each $i \in \V$,
	the Laplace-Stieltjes transform
	of the stationary matching time~$T_i$
	is defined as
	\begin{align} \label{eq:lst-def}
		\varphi_{T_i}(u)
		= \esp{e^{-u T_i}},
		\quad u \in \R_+.
	\end{align}
	One can show that
	this expectation is finite
	whenever $u \in \R_+$,
	but we will be especially interested
	in the range $0 \le u \le \alpha_i$.
	Proposition~\ref{prop:lst-rec} below
	gives a closed-form expression for
	this Laplace-Stieltjes transform
	by applying the distributional form
	of Little's law, in a similar way
	as in~\cite[Corollary~3.2]{CBL21}.
	
	\begin{proposition} \label{prop:lst-rec}
		For each $i \in \V$,
		the Laplace-Stieltjes transform
		of the matching time of class~$i$
		is given by
		\begin{align} \label{eq:lst-rec}
			\varphi_{T_i}(u)
			= g_{X}(z),
			\quad 0 \le u \le \alpha_i,
		\end{align}
		where the function $g_X$
		is defined on $\Z$ by~\eqref{eq:pgf-rec}
		and the vector $z \in \Z$
		is given by
		\begin{align*}
			z_j =
			\begin{cases}
				1 - \frac{u}{\alpha_i}
				&\text{if $j = i$,} \\
				1
				&\text{if $j \in \V \setminus \{i\}$.}
			\end{cases}
		\end{align*}
	\end{proposition}
	
	\begin{proof}
		Let $i \in \V$.
		The main argument consists of observing
		that class~$i$
		satisfies the assumptions of
		the following theorem,
		proved in \cite[Theorem~1]{KS88}:
		\begin{theorem}[Distributional form of Little's law]
			Let an ergodic queueing system
			be such that, for a given class~$i$ of customers:
			\begin{enumerate}[label=(\alph*)]
				\item \label{item:a}
				arrivals form
				a Poisson process
				with rate $\alpha_i$;
				\item \label{item:b}
				all arriving customers
				enter the system,
				and remain in the system until served,
				i.e.\ there is no blocking,
				balking, or reneging;
				\item \label{item:c}
				the customers leave
				the system one at a time
				in order of arrival;
				\item \label{item:d}
				for any time $t$,
				the arrival process
				after time~$t$
				and the time in the system
				of any customer
				arriving before~$t$
				are independent.
			\end{enumerate}
			Then the moment generating function $g_{X_i}$
			of the ergodic number $X_i$
			of customers of this class in the system
			and the Laplace-Stieltjes transform $\varphi_{T_i}$
			of the ergodic time $T_i$
			spent in the system
			by a customer of this class
			satisfy
			$g_{X_i}(z) = \varphi_{T_i}(\alpha_i (1 - z))$.
		\end{theorem}
		
		The distributional form
		of Little's law therefore implies that
		\begin{align*}
			\varphi_{T_i}(u)
			= g_{X_i}\left( 1 - \frac{u}{\alpha_i} \right),
			\quad 0 \le u \le \alpha_i,
		\end{align*}
		where $g_{X_i}$ is the probability generating function
		of the stationary number $X_i$
		of unmatched class-$i$ items.
		We conclude by observing that,
		for each $z_i \in [0, 1]$, we have
		$g_{X_i}(z_i) = g_X(z)$
		where $z \in \Z$ is the $N$-dimensional vector
		with $z_i$ in component~$i$
		and $1$ elsewhere.
	\end{proof}

	\section{Impact of load on performance} \label{sec:app}
	
	Performance in a stochastic matching model
	is the result of an intricate interaction
	between arrival rates and matching compatibilities.
	Even finding arrival rates that
	stabilize the matching model associated
	with a given compatibility graph
	is non-trivial \textit{a priori}.
	More subtly, it was observed in \cite{CDFB20} that,
	given a compatibility graph and a set of arrival rates
	that make the system stable,
	adding more edges to the compatibility graph
	may hurt performance
	(even if~\eqref{eq:stability} guarantees
	that the system will remain stable).
	Furthermore, applying the results of
	Sections~\ref{sec:oi}
	and~\ref{sec:perf}
	to derive performance
	in a real-world matching problem
	will often be impossible
	due to the complexity of the formulas.
	In this section, we show that
	these results are also useful to gain intuition about
	the choices of parameters that optimize performance
	or at least make the system stable.
	By analogy with queueing theory,
	we define the \emph{loads}
	of the independent sets as
	\begin{align} \label{eq:load}
		\rho(\I) = \frac{\alpha(\I)}{\alpha(\E(\I))},
		\quad \I \in \ind.
	\end{align}
	These loads will be key to analyze performance,
	as for instance
	the stability condition~\eqref{eq:stability}
	can be rewritten as
	$\rho(\I) < 1$ for each $\I \in \ind$.
	Although this definition
	is inspired from the analogy with queueing theory,
	there is a fundamental difference
	between~\eqref{eq:load} and
	the regular definition of load in a queue,
	namely in~\eqref{eq:load}
	the arrival rates appear both
	in the numerator and denominator.
	This implies in particular that
	the loads are not increasing functions
	of the arrival rates.
	
	\subsection{Heuristics to optimize performance} \label{subsec:heuristic}
	
	Given the conclusions of~\cite{CDFB20},
	one may wonder,
	for a given vector of arrival rates
	$\alpha = (\alpha_1, \alpha_2, \ldots, \alpha_N)$
	and a given compatibility graph
	that lead to a stable
	stochastic matching model,
	whether or not removing edges
	from the compatibility graph
	may improve performance,
	and in this case which edge(s) should be removed.
	This discrete optimization problem
	is difficult in general,
	as the number of configurations
	is exponential in the number of classes.
	To gain insight into this question,
	we consider a related problem
	that consists of finding arrival rates
	that stabilize the system
	and minimize the mean matching time
	for a given compatibility graph.
	Solving this optimization problem exactly
	with the formulas of Section~\ref{sec:perf}
	is practically unfeasible
	when the number of classes is large,
	but we propose two heuristics
	that yield acceptable performance in practice.
	
	\paragraph{Degree proportional}
	
	It was shown in \cite[Theorem~1]{MM16} that
	there exists a vector of arrival rates
	$\alpha = (\alpha_1, \alpha_2, \ldots, \alpha_N)$
	that stabilizes the system
	if and only if the compatibility graph is non-bipartite
	(which we assumed in Section~\ref{subsec:model}),
	but exhibiting such arrival rates
	is not trivial \emph{a priori}.
	Lemma~\ref{lem:weight} below
	gives a simple method to generate
	(an infinite number of)
	such arrival rates.
	
	\begin{lemma} \label{lem:weight}
		Let $A = (a_{i, j})_{i, j \in \V}$
		denote a symmetric matrix
		such that $a_{i, j} > 0$
		if and only if nodes $i$ and $j$
		are neighbors in the compatibility graph.
		For each $i \in \V$,
		the weight of class~$i$
		is defined as $d_i = \sum_{j \in \E_i} a_{i, j}$.
		The matching model is stable
		if the arrival rates are chosen
		to be proportional to the weights, that is,
		\begin{align} \label{eq:weight}
			\alpha_i
			= \frac{d_i}{\sum_{j \in \V} d_j},
			\quad i \in \V.
		\end{align}
	\end{lemma}
	
	\begin{proof}
		Assume that
		the arrival rates are given
		by~\eqref{eq:weight}
		for some matrix~$A$ that
		satisfies the assumptions
		of the proposition,
		and let $d = \sum_{i \in \V} d_i$.
		We will prove that these arrival rates
		satisfy the stability condition~\eqref{eq:stability}.
		Consider an independent set $\I \in \ind$.
		We have successively
		\begin{align*}
			d \cdot \alpha(\I)
			&= \sum_{i \in \I} d_i
			= \sum_{i \in \I} \sum_{j \in \E_i} a_{i,j}
			= \sum_{j \in \E(\I)}
			\sum_{i \in \I \cap \E_j} a_{j,i}
			< \sum_{j \in \E(\I)}
			\sum_{i \in \E_j} a_{j,i}
			= \sum_{j \in \E(\I)}
			d_j
			= d \cdot \alpha(\E(\I)).
		\end{align*}
		To prove that the above inequality is indeed strict,
		it suffices to show that we cannot find
		an independent set $\I \in \ind$ such that
		$\E(\E(\I)) = \I$.
		We can prove this by contradiction.
		Indeed, if there were such
		an independent set $\I$,
		this would imply that
		$\E(\I) = \V \setminus \I$
		(because the compatibility graph is connected)
		and that $\E(\I)$ is also an independent set,
		which would contradict our assumption
		that the compatibility graph is non-bipartite.
	\end{proof}
	
	The matrix~$A$ can be interpreted
	as the adjacency matrix
	of a weighted variant
	of the compatibility graph.
	As a special case,
	if $A$ is the adjacency matrix
	of the compatibility graph,
	then $d_i$ is the degree of class~$i$,
	in which case the arrival rates
	defined by~\eqref{eq:weight}
	are called \emph{degree proportional}.
	For example, in the graph
	of \figurename~\ref{fig:toy-graph},
	the degree-proportional arrival rates
	are $\alpha_1 = \alpha_2 = \frac14$,
	$\alpha_3 = \frac38$, and $\alpha_4 = \frac18$;
	choosing these arrival rates yields approximately
	a probability 0.17
	that the system is empty
	and a mean matching time
	of 2.25 time units.
	In general, if the number of classes
	is sufficiently small so that all independent sets
	\emph{can}
	be enumerated,
	this degree-proportional solution
	can also be used to initialize
	an optimization algorithm
	that searches for arrival rates
	that minimize the mean matching time
	by applying Proposition~\ref{prop:LI-rec}.
	
	\paragraph{Load minimization}
	
	The loads defined in~\eqref{eq:load}
	also play an instrumental role
	in better understanding which arrival rates
	yield acceptable performance.
	To understand this, let us first focus on
	the probability that the system is empty,
	which is the inverse of the normalization constant
	and is given by Proposition~\ref{prop:piI-rec}.
	First observe that~\eqref{eq:piI-rec}
	can be rewritten as
	\begin{align} \label{eq:piI-load}
		\pi(\I)
		&= \frac{\rho(\I)}{1 - \rho(\I)}
		\left(
		\sum_{i \in \I} \frac{\alpha_i}{\alpha(\I)}
		\pi(\I \setminus \{i\})
		\right),
		\quad \I \in \ind,
	\end{align}
	where the loads $\rho(\I)$ for $\I \in \ind$
	are given by~\eqref{eq:load}.
	Expanding the recursion yields
	\begin{align*}
		\pi(\I)
		&= \pi(\emptyset)
		\sum_{s \in \mathfrak{S}_\I}
		\left(
		\prod_{p = 1}^{|\I|}
		\frac{\rho(\{s_1, \ldots, s_p\})}
		{1 - \rho(\{s_1, \ldots, s_p\})}
		\right)
		\left(
		\prod_{p = 1}^{|\I|}
		\frac{\alpha_{s_p}}{\alpha(\{s_1, \ldots, s_p\})}
		\right),
		\quad \I \in \ind,
	\end{align*}
	where $\mathfrak{S}_\I$ is the set
	of permutations of the set $\I$,
	interpreted as the set of all sequences
	$s = (s_1, s_2, \ldots, s_{|\I|})$
	in which each element of $\I$
	appears exactly once.
	The first product may tend to infinity
	if the load $\rho(\J)$ of
	any set $\J \subseteq \I$ tends to one,
	while the second product is always
	between zero and one.
	By applying the normalization condition,
	we obtain that $1 / \pi(\emptyset)$
	is a linear combination of products of the form
	\begin{align*}
		\prod_{p = 1}^n
		\frac{\rho(\{s_1, \ldots, s_p\})}
		{1 - \rho(\{s_1, \ldots, s_p\})},
		\quad \I = \{s_1, s_2, \ldots, s_n\} \in \ind.
	\end{align*}
	Similarly, we can rewrite~\eqref{eq:LI-rec}
	as follows:
	\begin{align} \label{eq:LI-load}
		L(\I) \pi(\I)
		&= \frac{\pi(\I)}{1 - \rho(\I)}
		+ \frac{\rho(\I)}{1 - \rho(\I)}
		\left(
		\sum_{i \in \I}
		\frac{\alpha_i}{\alpha(\I)}
		L(\I \setminus \{i\}) \pi(\I \setminus \{i\})
		\right),
		\quad \I \in \ind.
	\end{align}
	In both cases,
	the variations of the function
	$x \mapsto \frac{x}{1 - x}$ on $(0, 1)$
	suggest that solving
	the following optimization problem
	is a good heuristic to maximize
	the probability that the system is empty
	and minimize the mean matching time:
	\begin{align} \label{eq:optimization}
		\begin{aligned}
			\underset{\alpha}{\text{Minimize}}
			\quad &
			\max_{\I \in \ind} \left( \rho(\I) \right), \\
			\text{subject to}
			\quad &
			\alpha_i \ge 0, \quad i \in \V, \\
			& \sum_{i \in \V} \alpha_i = 1.
		\end{aligned}
	\end{align}
	One can verify that
	a solution of this optimization problem satisfies
	the stability condition~\eqref{eq:stability}.
	In the example
	of \figurename~\ref{fig:toy-graph},
	the objective function of
	the optimization problem~\eqref{eq:optimization}
	can be rewritten as
	\begin{align*}
		\max\left(
		\frac{\alpha_1}{\alpha_2 + \alpha_3},
		\frac{\alpha_2}{\alpha_1 + \alpha_3},
		\frac{\alpha_3}{\alpha_1 + \alpha_2 + \alpha_4},
		\frac{\alpha_4}{\alpha_3},
		\frac{\alpha_1 + \alpha_4}{\alpha_2 + \alpha_3},
		\frac{\alpha_2 + \alpha_4}{\alpha_1 + \alpha_3}
		\right),
	\end{align*}
	and the unique solution is
	$\alpha_1 = \alpha_2 = \alpha_3 = \frac13$ and $\alpha_4 = 0$.
	These arrival rates happen to be a global maximum for the probability that the system is empty, with maximum value 0.25, and a local minimum for the mean matching time, with minimum value 1.5 time units\footnote{Although numerical evaluations suggest that this is also a global minimum, proving this conjecture is complicated by the structure of~\eqref{eq:LI-rec}. Better understanding the properties of the mean matching time seen as a function of the arrival rates and/or the graph structure would be an interesting topic for future work.}.
	
	In general,
	calculating a solution of
	the optimization problem~\eqref{eq:optimization}
	requires enumerating all independent sets,
	which is precisely why
	we may not be able to apply
	the results of Section~\ref{sec:perf}
	on the first place.
	This issue can be mitigated
	by focusing on the independent sets
	with cardinality at most~$n$ for some $n \ll N$,
	or on the independent sets
	containing the most \emph{central} nodes
	(under an appropriate definition
	of centrality),
	but the obtained solution
	may not be stable.
	To further stress the importance of
	minimizing the loads,
	we now consider a scaling regime
	in which the load of
	a (subset of) class(es) becomes close to one,
	thus making the system unstable.

	\subsection{Heavy-traffic regime} \label{subsec:heavy}
	
	We consider a scaling regime
	called \emph{heavy traffic}
	by analogy with the regime of the same name
	studied in queueing theory.
	In a nutshell,
	the stability condition
	associated with a maximal independent set
	becomes violated in~\eqref{eq:stability},
	so that items of these classes
	accumulate in the buffer
	(while items of other classes become scarce);
	our result describes
	the corresponding limiting distributions and performance metrics.
	The analysis relies on
	the closed-form expressions
	derived in Section~\ref{sec:perf}.
	The proof technique and the results
	generalize those of
	\cite[Theorem~3.1]{CBL21} but,
	as announced in the introduction,
	the conclusions are different
	and illustrate the fundamental difference
	between stochastic matching models
	and conventional queueing models.

	\paragraph{Scaling regime}
	
	Fix $\I \in \ind$, an independent set
	that is maximal with respect to the inclusion relation.
	We have in particular that
	$\E(\I) = \V \setminus \I$
	and
	$\alpha(\I) + \alpha(\E(\I))
	= \alpha(\I) + \alpha(\V \setminus \I)
	= 1$.
	We consider a scaling regime where,
	for each $i \in \V$,
	\begin{align} \label{eq:scaling}
		\alpha_i = \begin{cases}
			p_i \frac\rho{1 + \rho}, &i \in \I, \\
			q_i \frac1{1 + \rho}, &i \in \V \setminus \I,
		\end{cases}
	\end{align}
	with $\sum_{i \in \I} p_i
	= \sum_{i \in \V \setminus \I} q_i = 1$
	and $\rho \in (0, 1)$.
	In this way,
	$\rho = \alpha(\I) / \alpha(\V \setminus \I)
	= \rho(\I)$
	is the ratio of the arrival rate of the classes in $\I$
	to the arrival rate of their compatible classes,
	that is, the load of set $\I$.
	We will show that, under some technical assumptions,
	as $\rho \uparrow 1$,
	the buffer becomes saturated with
	items of the classes in $\I$
	while running out of items of the classes
	in $\V \setminus \I$.
	That the classes in $\I$ become saturated
	is a consequence of~\eqref{eq:scaling},
	as we have
	\begin{align} \label{eq:saturation}
		\lim_{\rho \uparrow 1} \alpha(\I)
		= \lim_{\rho \uparrow 1} \alpha(\V \setminus \I)
		= \frac12.
	\end{align}
	The following technical assumption guarantees that
	the set $\I$ is the only one that becomes saturated.
	
	\begin{assumption} \label{ass:stability}
		There exists $\rho^* \in [0, 1)$ such that
		the stability conditions~\eqref{eq:stability}
		are satisfied by the measures $\alpha$
		defined by~\eqref{eq:scaling}
		for each $\rho \in (\rho^*, 1)$, and
		\begin{align*}
			\lim_{\rho \uparrow 1} \alpha(\J)
			< \lim_{\rho \uparrow 1} \alpha(\E(\J)),
			\quad \J \in \ind \setminus \{\I\}.
		\end{align*}
	\end{assumption}
	
	\paragraph{Asymptotic distribution}
	
	Proposition~\ref{prop:heavy} below shows that,
	under the above technical assumption,
	the mean number of items
	of the classes in $\I$
	goes to infinity
	like $(1 - \rho)^{-1}$
	as $\rho \uparrow 1$,
	while the mean number of items
	of the classes in $\V \setminus \I$
	goes to zero.
	Furthermore, the distribution of items
	of the classes in $\I$
	depends on their respective arrival rates
	but not on their compatibility constraints
	with the classes in $\V \setminus \I$.
	Similarly, as intuition suggests,
	the waiting probability of
	the classes in~$\I$ tends to one
	as $\rho \uparrow 1$,
	while the waiting probability of
	the classes in~$\V \setminus \I$
	tends to zero.
	In this way,
	despite the structural differences
	between the stochastic matching model
	and a more conventional queueing model,
	the classes in $\I$ experience
	the same performance as
	in an M/M/1 multi-class queue
	in which the load $\rho$ tends to one
	while, for each $i \in \I$,
	the relative arrival rate
	of class~$i$ is kept constant equal to $p_i$.
	The fundamental difference
	with an M/M/1 multi-class queue,
	discussed in the remark below,
	is that the number of items
	of the classes in $\V \setminus \I$
	tends to zero.
	Given $\rho^* \in [0, 1)$
	and two functions $f$ and $g$,
	defined on $(\rho^*, 1)$, such that
	$g(\rho) > 0$ for each $\rho \in (\rho^*, 1)$,
	we write $f(\rho) \underset{\rho \uparrow 1}{\sim} g(\rho)$
	if $f(\rho) / g(\rho) \xrightarrow[\rho \uparrow 1]{} 1$.
	
	\begin{proposition} \label{prop:heavy}
		If Assumption~\ref{ass:stability}
		is satisfied, then the following
		limiting results hold:
		\begin{enumerate}[label=(14-\alph*),
			leftmargin=\widthof{~(14-a)~}]
			\item \label{item:heavy-piI}
			The  stationary distribution
			of the set of unmatched item classes
			satisfies:
			\begin{align*}
				\pi(\I) &\xrightarrow[\rho \uparrow 1]{} 1,
				&
				\pi(\J) &\xrightarrow[\rho \uparrow 1]{} 0,
				\quad \J \in \inde \setminus \{\I\}.
			\end{align*}
			\item \label{item:heavy-wi}
			The waiting probability
			of each item class satisfies:
			\begin{align*}
				\omega_i &\xrightarrow[\rho \uparrow 1]{} 1,
				\quad i \in \I,
				&
				\omega_i &\xrightarrow[\rho \uparrow 1]{} 0,
				\quad i \in \V \setminus \I.
			\end{align*}
			\item \label{item:heavy-Li}
			The mean number of
			unmatched items of each class satisfies:
			\begin{align*}
				L_i = \esp{X_i}
				&\underset{\rho \uparrow 1}{\sim}
				p_i \frac{\rho}{1 - \rho},
				\quad i \in \I,
				&
				L_i = \esp{X_i}
				&\xrightarrow[\rho \uparrow 1]{} 0,
				\quad i \in \V \setminus \I.
			\end{align*}
			In particular, the limiting mean number of
			unmatched items (over all classes) satisfies
			$L \underset{\rho \uparrow 1}{\sim}
			\frac{\rho}{1 - \rho}$.
			\item \label{item:heavy-X}
			The number of unmatched items of each class
			has the following limit in distribution:
			\begin{align*}
				(1 - \rho) (X_i)_{i \in \I}
				&\xrightarrow[\rho \uparrow 1]{d}
				\expo(1) (p_i)_{i \in \I},
				&
				(X_i)_{i \in \V \setminus \I}
				&\xrightarrow[\rho \uparrow 1]{d}
				0,
			\end{align*}
			where $\expo(1)$ represents a unit-mean
			exponentially-distributed random variable.
			\item \label{item:heavy-ETi}
			The mean matching time of
			each class satisfies:
			\begin{align*}
				\esp{T_i}
				&\underset{\rho \uparrow 1}{\sim}
				\frac2{1 - \rho},
				\quad i \in \I,
				&
				\esp{T_i}
				&\xrightarrow[\rho \uparrow 1]{} 0,
				\quad i \in \V \setminus \I.
			\end{align*}
			\item \label{item:heavy-Ti}
			The matching time of each class
			has the following limit in distribution:
			\begin{align*}
				(1 - \rho) T_i
				&\xrightarrow[\rho \uparrow 1]{d}
				\expo\left( \frac12 \right),
				\quad i \in \I,
				&
				T_i 
				&\xrightarrow[\rho \uparrow 1]{d}
				0,
				\quad i \in \V \setminus \I,
			\end{align*}
			where $\expo(\frac12)$ represents
			an exponentially-distributed random variable
			with rate $\frac12$.
		\end{enumerate}
	\end{proposition}
	
	\begin{proof}[Sketch of proof]
		The proof of each property relies
		on the corresponding expression
		derived in Section~\ref{sec:perf}.
		Each proof follows the same pattern:
		the expression for which
		we want to calculate the limit is a fraction
		whose numerator and denominator are
		sums of terms given by recursive expressions;
		we first use the recursive expressions
		to identify the dominating term in each sum,
		and then we simplify the dominating terms
		to derive the limit.
		The complete proof is given
		in the appendix.
	\end{proof}
	
	\begin{remark*}
		The analysis is simplified
		by our choice of a maximal independent set
		and by Assumption~\ref{ass:stability}.
		Indeed, these two assumptions guarantee that
		$\I$ is the only independent set
		that becomes saturated as $\rho \uparrow 1$,
		and that all other classes are absorbed by
		the surplus of items in the classes in $\I$.
		We could also consider a scaling regime
		where these assumptions are not satisfied.
		For instance, we could choose
		an independent set $\I$
		that is not maximal with respect
		to the inclusion relation.
		Under some technical assumptions,
		we would obtain a scaling regime where
		the buffer becomes saturated with
		items of the classes in $\I$ and
		runs out of items of the classes in $\E(\I)$,
		while containing a positive but finite
		number of items of the classes
		in $\V \setminus (\I \cup \E(\I))$.
		This intermediary scaling regime
		becomes apparent in the numerical results
		of Section~\ref{sec:num}.
	\end{remark*}
	
	\begin{remark*}
		The above-defined heavy-traffic regime
		is actually closer to
		a \emph{light-traffic} regime
		for the classes in $\V \setminus \I$,
		as Proposition~\ref{prop:heavy}
		shows that the number of items
		of these classes tends to zero.
		Yet, one can verify that the loads
		of the classes of this set do
		not tend to zero.
		These two observations may seem paradoxical,
		especially with conventional queueing models,
		like the M/M/1 queue, in mind.
		This apparent paradox
		can be explained by recalling that
		the arrival of an item
		can either increase \emph{or decrease}
		the number of present items.
		This is different from
		conventional queueing models,
		in which the arrival of a customer
		always increases the number of present customers.
		Intuitively, in the scaling regime
		studied in this section,
		the arrival of an item
		of a class in $\I$
		will (almost) always
		increase the number of unmatched items,
		while the arrival of an item
		of a class in $\V \setminus \I$
		will (almost) always
		decrease the number of unmatched items.
	\end{remark*}

	\section{Numerical results} \label{sec:num}
	
	In this section,
	we study the toy examples
	of \figurename~\ref{fig:num-toy}
	and evaluate two performance metrics,
	the waiting probability
	and the mean matching time,
	defined in Sections~\ref{subsec:perf-waiting-proba} and~\ref{subsec:perf-mean-number}.
	All numerical results are computed
	using the closed-form expressions
	derived in these two sections.
	As before, we impose that
	$\sum_{i \in \V} \alpha_i = 1$,
	so that $\alpha_i$ is the long-term fraction of items
	that belong to class~$i$, for each $i \in \V$.
	More specifically, in each case,
	we impose that
	$\alpha_2 = \alpha_3 = \ldots = \alpha_N$
	and let the arrival rate $\alpha_1$ of class~$1$ vary
	such that its load~$\rho = \alpha_1 / \alpha(\E_1)$
	belongs to the interval $(0, 1)$.
	This allows us to gain insight into the impact of
	varying the load of one class on performance
	and to illustrate the heavy-traffic result
	of Section~\ref{subsec:heavy}.
	
	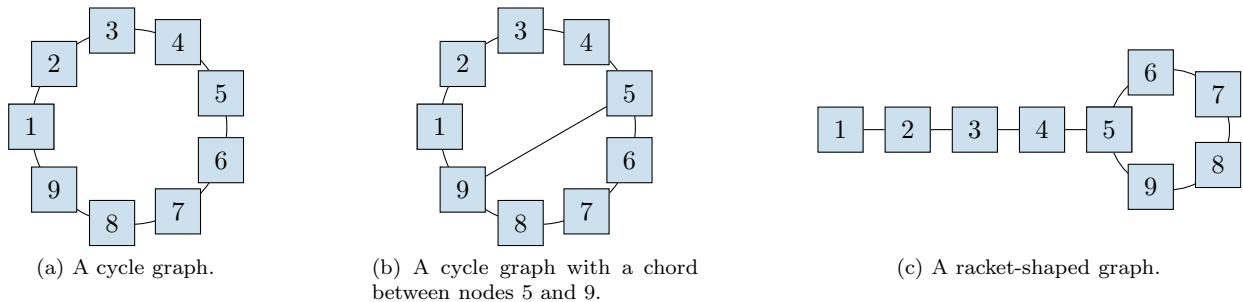
\begin{figure}[b]
		\centering
		\subfloat[A cycle graph.
		\label{fig:cycle-graph}]{%
			\begin{tikzpicture}
				\def\n{9}
				\def\radius{1.3cm}
				
				\foreach \s in {1,...,\n}
				{
					\draw[-]
					({180 - 360/\n * (\s - 1)}:\radius)
					arc
					({180 - 360/\n * (\s - 1)}
					:{180 - 360/\n * (\s)}:\radius);
					\node[smallclass] (\s)
					at ({180 - 360/\n * (\s - 1)}:\radius)
					{$\s$};
				}
				\node[smallclass]
				at ({180}:\radius)
				{$1$};
			\end{tikzpicture}
		}
		\hfill
		\subfloat[A cycle graph
		with a chord between nodes~$5$ and~$9$.
		\label{fig:chord-graph}]{%
			\begin{tikzpicture}
				\def\n{9}
				\def\radius{1.3cm}
				
				\foreach \s in {1,...,\n}
				{
					\draw[-]
					({180 - 360/\n * (\s - 1)}:\radius)
					arc
					({180 - 360/\n * (\s - 1)}
					:{180 - 360/\n * (\s)}:\radius);
					\node[smallclass] (\s)
					at ({180 - 360/\n * (\s - 1)}:\radius)
					{$\s$};
				}
				\node[smallclass] at ({180}:\radius) {$1$};
				\draw[-] (5) -- (9);
				
				\node at ($(1.west)-(.48cm,0)$) {};
				\node (test) at ({0}:\radius) {};
				\node at ($(test.east)+(.48cm,0)$) {};
			\end{tikzpicture}
		}
		\hfill
		\subfloat[A racket-shaped graph.
		\label{fig:racket-graph}]{%
			\begin{tikzpicture}
				\def\n{5}
				\def\nminusone{4}
				\def\m{9}
				\def\radius{.8cm}
				\def\delta{1.2cm}
				
				\foreach \s in {\n,...,\m} {
					\pgfmathsetmacro{\i}{\s + \n}
					\draw[-]
					({180 - 360/\n * (\s - \n)}:\radius)
					arc
					({180 - 360/\n * (\s - \n)}
					:{180 - 360/\n * (\s - \n + 1)}:\radius);
					\node[smallclass] (\s) at
					({180 - 360/\n * (\s - \n)}:\radius)
					{$\s$};
				}
				\node[class]
				at ({180}:\radius)
				{$5$};
				
				\foreach \s in {\nminusone, ..., 1} {
					\pgfmathsetmacro{\t}{\s + 1}
					\node[smallclass] (\s)
					at ($(\t)-(\delta, 0)$) {\s};
				}
				
				\draw[-] (1) -- (2);
				\draw[-] (2) -- (3);
				\draw[-] (3) -- (4);
				\draw[-] (4) -- (5);
				
				\node at ($(\m)-(0,.67cm)$) {};
			\end{tikzpicture}
		}
		\caption{Toy examples with $N = 9$ classes.}
		\label{fig:num-toy}
	\end{figure}

	\subsection{Cycle graph} \label{subsec:num-cycle}
	
	We first consider a stochastic matching model
	with a simple yet insightful compatibility graph,
	namely a cycle graph with an odd number of nodes.
	More specifically,
	we let $N = 2K + 1$ denote the number of classes,
	where $K \in \{1, 2, 3, \ldots\}$,
	and we assume that, for each $i \in \V$,
	items of classes~$i$ and $i+1$ can be matched
	with one another
	(with the convention that $i+1 = 1$ if $i = N$).
	An example of a cycle graph
	with $N = 9$ nodes
	is shown in \figurename~\ref{fig:cycle-graph}.
	The dynamics in a matching model
	with a cycle graph are representative
	of the dynamics obtained with
	more complex compatibility graphs.
	Indeed, as recalled in Section~\ref{subsec:stability},
	a stochastic matching model
	can be stabilized if and only if
	its compatibility graph is non-bipartite,
	which means that it contains
	at least one odd cycle.
	The arrival rates are given by
	\begin{align} \label{eq:circle-load}
		\alpha_1 &= \frac{\frac\rho{K}}{1 + \frac\rho{K}},
		& \alpha_i &= \frac{\frac1{2K}}{1 + \frac\rho{K}},
		\quad i \in \{2, 3, \ldots, N\},
	\end{align}
	so that $\rho$ is the load of class~$1$.
	One can show by a direct reasoning\footnote{%
		Let $\I \in \ind$.
		The sets
		$\I + 1 = \{i + 1, i \in \I\}$
		and $\I - 1 = \{i - 1, i \in \I\}$
		(again with the convention that
		$i + 1 = 1$ if $i = N$
		and $i - 1 = N$ if $i = 1$)
		both contain $|\I|$ nodes
		and are included into $\E(\I)$.
		It suffices to show that
		at least one of these two sets
		contains an element that the other does not.
		In turn, to prove this,
		it suffices to observe that
		an independent set $\I$ contains
		at most $K$ nodes,
		so that at least two of these nodes
		are separated by at least two nodes
		that do not belong to $\I$.%
	}
	that $|\E(\I)| \ge |\I| + 1$
	for each $\I \in \ind$,
	which is sufficient to prove that,
	under the arrival rates~\eqref{eq:circle-load},
	the system is stable for each $\rho \in (0, 1)$.
	In what follows, we focus on the case
	$N = 9$ for simplicity of exposition,
	but the results for
	$N \in \{3, 5, 7, 11, 13, 15\}$
	(not shown here) are similar.
	To illustrate the performance gain
	permitted by the recursive expressions
	of Section~\ref{sec:perf}
	compared to existing expressions,
	we observe that, again with $N = 9$,
	the number of terms to sum
	to derive each metric
	is equal to 75 using dynamic programming
	and to 459 if we naively apply the formulas of
	\cite[Equation~(5)]{MBM21}
	and \cite[Proposition~4]{CDFB20},
	giving a ratio of
	$459/75 \simeq 6.12$;
	with $N = 15$ classes,
	the numbers of terms to sum
	are $1,363$ and $234,405$,
	respectively, giving a ratio of
	$234,405/1,363 \simeq 171.98$.
	
	\begin{figure}[ht]
		\centering
		\hspace{.8cm}
		\begin{tikzpicture}
			\begin{axis}[probaplotstyle, hide axis]
				\addlegendimage{teal, no markers}
				\addlegendentry{Overall};
				
				\addlegendimage{orange, densely dashed, no markers}
				\addlegendentry{1};
				
				\addlegendimage{green, densely dashdotted, no markers}
				\addlegendentry{2 and 9};
				
				\addlegendimage{red, densely dotted, no markers,}
				\addlegendentry{3 and 8};
				
				\addlegendimage{purple, dashed, no markers}
				\addlegendentry{4 and 7};
				
				\addlegendimage{olive, dashdotted, no markers}
				\addlegendentry{5 and 6};
			\end{axis}
		\end{tikzpicture}
		\\
		\pgfplotstableread{cycle.csv}\model
		\begin{tikzpicture}
			\begin{axis}[probaplotstyle]
				
				\addplot+[
				teal, no markers,
				] table[x=load, y=waiting probability 9]{\model};
				
				\addplot+[
				orange, densely dashed, no markers,
				] table[x=load, y=waiting probability 0]{\model};
				
				\addplot+[
				green, densely dashdotted, no markers,
				] table[x=load, y=waiting probability 1]{\model};
				
				\addplot+[
				red, densely dotted, no markers,
				] table[x=load, y=waiting probability 2]{\model};
				
				\addplot+[
				purple, dashed, no markers,
				] table[x=load, y=waiting probability 3]{\model};
				
				\addplot+[
				olive, dashdotted, no markers,
				] table[x=load, y=waiting probability 4]{\model};
				
			\end{axis}
		\end{tikzpicture}
		\hfill
		\begin{tikzpicture}
			\begin{axis}[meanplotstyle]
				
				\addplot+[
				teal, no markers,
				] table[x=load, y=mean waiting time 9]{\model};
				
				\addplot+[
				orange, densely dashed, no markers,
				] table[x=load, y=mean waiting time 0]{\model};
				
				\addplot+[
				green, densely dashdotted, no markers,
				] table[x=load, y=mean waiting time 1]{\model};
				
				\addplot+[
				red, densely dotted, no markers,
				] table[x=load, y=mean waiting time 2]{\model};
				
				\addplot+[
				purple, dashed, no markers,
				] table[x=load, y=mean waiting time 3]{\model};
				
				\addplot+[
				olive, dashdotted, no markers,
				] table[x=load, y=mean waiting time 4]{\model};
				
			\end{axis}
		\end{tikzpicture}
		\caption{Performance (overall and per class) in a cycle
			with $N = 9$ classes.}
		\label{fig:cycle-num}
	\end{figure}
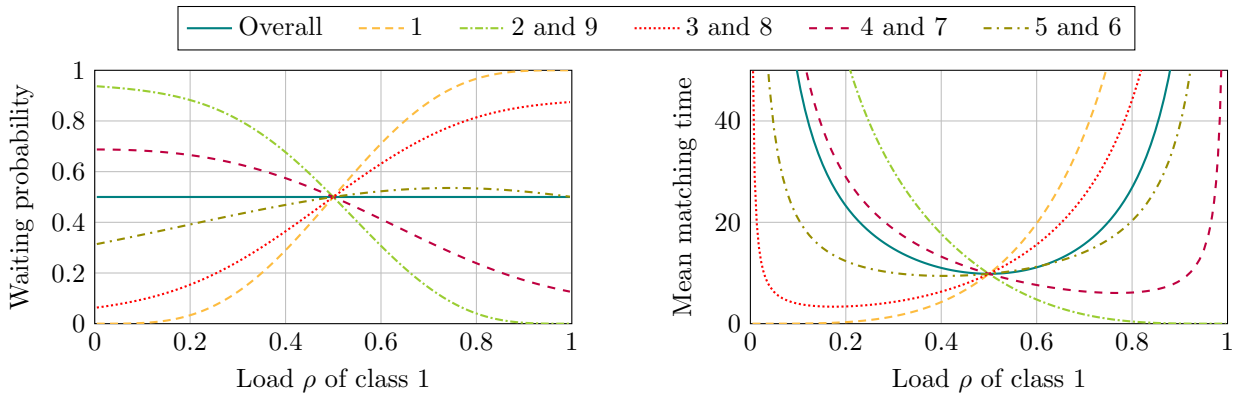
	
	The waiting probability
	and mean matching time
	are shown in \figurename~\ref{fig:cycle-num}.
	Because of the symmetry of the model,
	the performance of a class only depends on
	its distance to class~$1$, so that
	classes~$i + 1$ and $N - i + 1$
	have the same performance
	for each $i \in \{1, 2, \ldots, K\}$.
	Focusing on the waiting probability,
	we can partition classes into two categories,
	those at an odd distance from class~$1$
	(namely, classes 2, 4, 7, and 9)
	and those at an even distance from class~$1$
	(namely, classes 1, 3, 5, 6, and 8).
	Indeed, as the load of class~$1$ increases,
	the waiting probability
	of the former set of classes decreases
	while that of the latter set of classes increases.
	The only exception is
	the waiting probability of classes~$5$ and $6$,
	which first increases and then decreases;
	we will mention this again at the end of this paragraph.
	Intuitively, increasing the load of class~$1$
	has the following cascading effect.
	To preserve the system stability,
	the additional class-$1$ items
	need be matched with items
	of classes~$2$ and $9$,
	so that the waiting probability
	of these two classes decreases.
	The side effect is that
	there are fewer items of classes~$2$ and $9$
	to be matched with items
	of classes~$3$ and $8$,
	so that the waiting probability
	of classes~$3$ and $8$ increases.
	In turn,
	surplus items of classes~$3$ and $8$
	need be matched with
	items of classes~$4$ and $7$
	to preserve stability,
	so that the waiting probability
	of classes~$4$ and $7$ decreases,
	and so on.
	The only exception is classes~$5$ and $6$,
	whose waiting probability
	is first increasing and then decreasing.
	Focusing on class~$5$ for instance,
	the non-monotonic evolution
	of the waiting probability
	is the result of two conflicting effects.
	On the one hand, the even path
	$1$--$2$--$3$--$4$--$5$
	suggests that
	increasing the load of class~$1$
	tends to increase
	the waiting probability of class~$5$.
	On the other hand,
	the odd path
	$1$--$9$--$8$--$7$--$6$--$5$
	suggests that increasing the load of class~$1$
	tends to decrease the waiting probability of class~$5$.
	
	Considering the mean matching time
	allows us to refine the above discussion.
	Indeed, we can observe that,
	except for classes $1$, $2$, and $9$,
	the mean matching time of all classes
	tends to infinity both
	when $\rho$ tends to zero
	and when $\rho$ tends to one.
	This suggests that
	not only the performance
	of classes~$5$ and $6$
	but also those of other classes
	are the result of two opposite effects,
	one that is conveyed by
	the odd-length path from class~$1$
	and the other by
	the even-length path from class~$1$.
	
	To better understand the speeds at which
	the mean matching times of classes tend to infinity,
	let us start with $\rho = \frac12$
	(meaning that all arrival rates are equal)
	and consider the chain reaction
	on the path
	$1$--$2$--$3$--$\cdots$-$8$--$9$
	as the load $\rho$ decreases
	and tends to 1.
	The first class to be impacted is class~$2$,
	as this class needs the presence of class~$1$
	to remain stable,
	and indeed the mean matching time
	of class~$2$ tends to infinity the fastest.
	The side effect is that
	class-$3$ items are
	matched more frequently with class-$2$ items,
	so that the mean matching time of class~$3$ decreases.
	This implies that
	class-$3$ items are less available
	for class~$4$,
	so that the mean matching time
	of class~$4$ is the next
	to tend to infinity.
	By repeating this reasoning,
	we obtain that the mean matching time
	of class~$6$ is the next to tend to infinity,
	followed by the mean matching time of class~$8$.
	A similar phenomenon occurs
	on the path
	$1$--$9$--$8$--$\cdots$--$3$--$2$,
	so that the mean matching times
	of classes~$9$, $7$, $5$, and $3$ also
	tend to infinity one after the other.

	\subsection{Cycle graph with a chord}
	
	To complement our observations,
	we break the model symmetry by
	adding a chord between classes~$5$ and $9$.
	The arrival rates are still
	given by~\eqref{eq:circle-load}.
	It follows from~\eqref{eq:stability}
	that adding an edge to the compatibility graph
	cannot reduce the stability region
	of the corresponding stochastic matching model,
	so that the model is still stable
	for each $\rho \in (0, 1)$.
	We again focus on a model with $N = 9$ classes. We have verified that the reasoning we develop below generalizes
	to other numbers of classes and other placements of the chord relative to class~1, although the results vary qualitatively.
	
	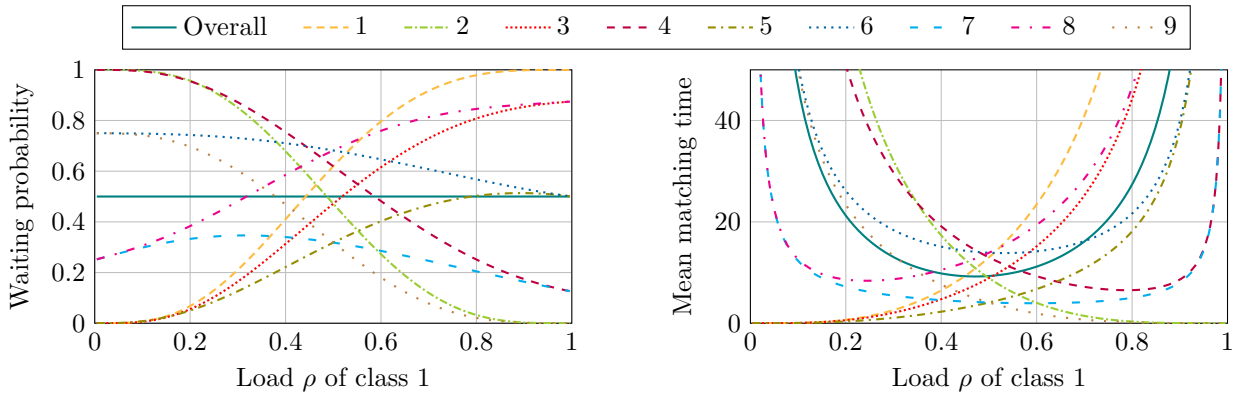
\begin{figure}[b]
		\centering
		\hspace{.8cm}
		\begin{tikzpicture}
			\begin{axis}[probaplotstyle, hide axis]
				\addlegendimage{teal, no markers}
				\addlegendentry{Overall};
				
				\addlegendimage{orange, densely dashed, no markers}
				\addlegendentry{1};
				
				\addlegendimage{green, densely dashdotted, no markers}
				\addlegendentry{2};
				
				\addlegendimage{red, densely dotted, no markers,}
				\addlegendentry{3};
				
				\addlegendimage{purple, dashed, no markers}
				\addlegendentry{4};
				
				\addlegendimage{olive, dashdotted, no markers}
				\addlegendentry{5};
				
				\addlegendimage{blue, dotted, no markers}
				\addlegendentry{6};
				
				\addlegendimage{cyan, loosely dashed, no markers}
				\addlegendentry{7};
				
				\addlegendimage{magenta, loosely dashdotted, no markers}
				\addlegendentry{8};
				
				\addlegendimage{brown, loosely dotted, no markers}
				\addlegendentry{9};
			\end{axis}
		\end{tikzpicture}
		\\
		\pgfplotstableread{chord.csv}\model
		\begin{tikzpicture}
			\begin{axis}[probaplotstyle,
				]
				
				\addplot+[
				teal, no markers,
				] table[x=load, y=waiting probability 9]{\model};
				
				\addplot+[
				orange, densely dashed, no markers,
				] table[x=load, y=waiting probability 0]{\model};
				
				\addplot+[
				green, densely dashdotted, no markers,
				] table[x=load, y=waiting probability 1]{\model};
				
				\addplot+[
				red, densely dotted, no markers,
				] table[x=load, y=waiting probability 2]{\model};
				
				\addplot+[
				purple, dashed, no markers,
				] table[x=load, y=waiting probability 3]{\model};
				
				\addplot+[
				olive, dashdotted, no markers,
				] table[x=load, y=waiting probability 4]{\model};
				
				\addplot+[
				blue, dotted, no markers,
				] table[x=load, y=waiting probability 5]{\model};
				
				\addplot+[
				cyan, loosely dashed, no markers,
				] table[x=load, y=waiting probability 6]{\model};
				
				\addplot+[
				magenta, loosely dashdotted, no markers,
				] table[x=load, y=waiting probability 7]{\model};
				
				\addplot+[
				brown, loosely dotted, no markers,
				] table[x=load, y=waiting probability 8]{\model};
			\end{axis}
		\end{tikzpicture}
		\hfill
		\begin{tikzpicture}
			\begin{axis}[meanplotstyle]
				
				\addplot+[
				teal, no markers,
				] table[x=load, y=mean waiting time 9]{\model};
				
				\addplot+[
				orange, densely dashed, no markers,
				] table[x=load, y=mean waiting time 0]{\model};
				
				\addplot+[
				green, densely dashdotted, no markers,
				] table[x=load, y=mean waiting time 1]{\model};
				
				\addplot+[
				red, densely dotted, no markers,
				] table[x=load, y=mean waiting time 2]{\model};
				
				\addplot+[
				purple, dashed, no markers,
				] table[x=load, y=mean waiting time 3]{\model};
				
				\addplot+[
				olive, dashdotted, no markers,
				] table[x=load, y=mean waiting time 4]{\model};
				
				\addplot+[
				blue, dotted, no markers,
				] table[x=load, y=mean waiting time 5]{\model};
				
				\addplot+[
				cyan, loosely dashed, no markers,
				] table[x=load, y=mean waiting time 6]{\model};
				
				\addplot+[
				magenta, loosely dashdotted, no markers,
				] table[x=load, y=mean waiting time 7]{\model};
				
				\addplot+[
				brown, loosely dotted, no markers,
				] table[x=load, y=mean waiting time 8]{\model};
				
			\end{axis}
		\end{tikzpicture}
		\caption{Performance (overall and per class) in a cycle
			with $N = 9$ classes
			supplemented with a chord
			between classes $5$ and $9$.}
		\label{fig:chord-num}
	\end{figure}
	
	The results are shown in
	\figurename~\ref{fig:chord-num}.
	Adding a chord breaks
	the symmetry between classes,
	but we can still divide classes
	into several categories
	depending on the limiting behavior
	of their waiting probability
	and mean matching time
	as the load $\rho$ tends to either zero or one.
	Let us first focus on
	the limiting regime where
	the load $\rho$ tends to zero.
	In view of the mean matching times,
	the first class to become unstable
	is class~$2$,
	which is natural because
	the supply rate of class~$3$ items
	is not sufficient to maintain its stability.
	The main difference with the previous scenario
	is that the stability of class~$9$
	is temporarily preserved by its compatibility
	with classes~$8$ and~$5$.
	The overabundance of class~$2$
	leads to the exhaustion of class~$3$,
	which in turn leads to the overabundance of class~$4$,
	which in turn leads to the exhaustion of class~$5$.
	Once class~$5$ is exhausted, the next classes
	to become unstable are its neighbors,
	classes~$6$ and $9$, which eventually leads to
	the overabundance of classes~$7$ and~$8$ as well.
	The limiting behavior as the load $\rho$
	tends to one is close to that observed
	for the cycle without chord.
	The reason is that the first immediate impact
	of increasing the load of class~$1$
	is the exhaustion of class~$9$ (as well as class~$2$),
	rendering the presence of the chord useless,
	so that then we are back to
	the scenario of Section~\ref{subsec:num-cycle}.

	\subsection{Racket-shaped graph}
	
	We finally consider the racket-shaped graph
	of \figurename~\ref{fig:racket-graph}.
	Although a line alone cannot be stabilized
	because it forms a bipartite graph,
	this example gives us an idea of
	the dynamics that may arise in such a graph.
	One can also think of this graph as
	being obtained by removing
	the edge between nodes $1$ and $9$
	in the cycle graph with a chord
	of \figurename~\ref{fig:chord-graph}.
	The arrival rates of classes are given by
	\begin{align*}
		\alpha_1&= \frac{\frac\rho{N-1}}{1 + \frac\rho{N-1}},
		&
		\alpha_i &= \frac{\frac1{N-1}}{1 + \frac\rho{N-1}},
		\quad i \in \{2, 3, \ldots, N\},
	\end{align*}
	so that $\rho$ is again the load of class~$1$.
	One can do a similar reasoning
	as for the cycle to show that
	this system is stable
	whenever $\rho \in (0, 1)$.
	We again focus on a scenario
	with $N = 9$ nodes in total,
	$4$ of which belong to the line part
	and $5$ to the cycle part.
	In general, we observed that
	the qualitative behavior of the results
	depends mainly on the parity of
	the numbers of classes
	on the line part
	and on the circle part of the graph,
	and we may again adapt the reasoning below
	to explain this qualitative behavior.
	
	The results for the racket-shaped graph
	of \figurename~\ref{fig:racket-graph}
	are shown in \figurename~\ref{fig:racket-num}.
	Along the line part of the graph,
	the impact of increasing or decreasing
	the load $\rho$ of class~$1$ is as one could expect.
	For instance, as $\rho$ tends to one,
	class-$2$ items become scarce,
	which leads to the instability of class~$3$,
	which in turn leads to the scarcity of class~$4$.
	Compared to the cycle with a chord,
	the most striking observation is that
	classes 5, 6, 7, 8, and 9
	remain stable as the load~$\rho$ tends to one,
	and that these classes have the same asymptotic
	waiting probability and mean matching time.
	The reason is that, asymptotically,
	the model restricted to
	classes~$5$ to $9$ evolves like
	a cycle with five classes
	that all have the same arrival rate,
	and this model is stable.
	
	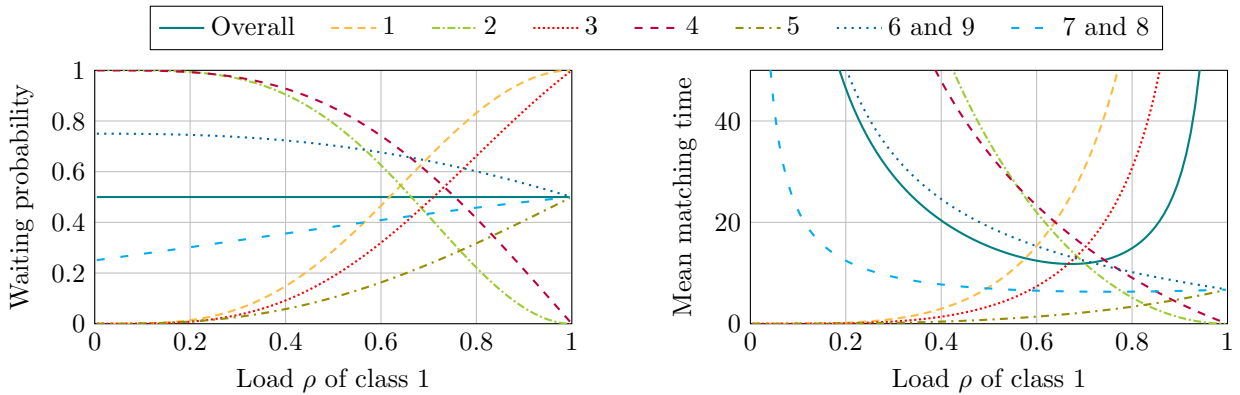
\begin{figure}[b]
		\centering
		\hspace{.8cm}
		\begin{tikzpicture}
			\begin{axis}[probaplotstyle, hide axis]
				\addlegendimage{teal, no markers}
				\addlegendentry{Overall};
				
				\addlegendimage{orange, densely dashed, no markers}
				\addlegendentry{1};
				
				\addlegendimage{green, densely dashdotted, no markers}
				\addlegendentry{2};
				
				\addlegendimage{red, densely dotted, no markers,}
				\addlegendentry{3};
				
				\addlegendimage{purple, dashed, no markers}
				\addlegendentry{4};
				
				\addlegendimage{olive, dashdotted, no markers}
				\addlegendentry{5};
				
				\addlegendimage{blue, dotted, no markers}
				\addlegendentry{6 and 9};
				
				\addlegendimage{cyan, loosely dashed, no markers}
				\addlegendentry{7 and 8};
			\end{axis}
		\end{tikzpicture}
		\\
		\pgfplotstableread{racket.csv}\model
		\begin{tikzpicture}
			\begin{axis}[probaplotstyle,
				legend columns=10,
				]
				
				\addplot+[
				teal, no markers,
				] table[x=load, y=waiting probability 9]{\model};
				
				\addplot+[
				orange, densely dashed, no markers,
				] table[x=load, y=waiting probability 0]{\model};
				
				\addplot+[
				green, densely dashdotted, no markers,
				] table[x=load, y=waiting probability 1]{\model};
				
				\addplot+[
				red, densely dotted, no markers,
				] table[x=load, y=waiting probability 2]{\model};
				
				\addplot+[
				purple, dashed, no markers,
				] table[x=load, y=waiting probability 3]{\model};
				
				\addplot+[
				olive, dashdotted, no markers,
				] table[x=load, y=waiting probability 4]{\model};
				
				\addplot+[
				blue, dotted, no markers,
				] table[x=load, y=waiting probability 5]{\model};
				
				\addplot+[
				cyan, loosely dashed, no markers,
				] table[x=load, y=waiting probability 6]{\model};
				
			\end{axis}
		\end{tikzpicture}
		\hfill
		\begin{tikzpicture}
			\begin{axis}[meanplotstyle]
				
				\addplot+[
				teal, no markers,
				] table[x=load, y=mean waiting time 9]{\model};
				
				\addplot+[
				orange, densely dashed, no markers,
				] table[x=load, y=mean waiting time 0]{\model};
				
				\addplot+[
				green, densely dashdotted, no markers,
				] table[x=load, y=mean waiting time 1]{\model};
				
				\addplot+[
				red, densely dotted, no markers,
				] table[x=load, y=mean waiting time 2]{\model};
				
				\addplot+[
				purple, dashed, no markers,
				] table[x=load, y=mean waiting time 3]{\model};
				
				\addplot+[
				olive, dashdotted, no markers,
				] table[x=load, y=mean waiting time 4]{\model};
				
				\addplot+[
				blue, dotted, no markers,
				] table[x=load, y=mean waiting time 5]{\model};
				
				\addplot+[
				cyan, loosely dashed, no markers,
				] table[x=load, y=mean waiting time 6]{\model};
				
			\end{axis}
		\end{tikzpicture}
		\caption{Performance (overall and per class) in the racket-shaped graph
			with $N = 9$ classes
			shown in \figurename~\ref{fig:racket-graph}.}
		\label{fig:racket-num}
	\end{figure}
	
	If we again focus on the asymptotic regime
	where the load $\rho$ tends to one,
	it seems that the overall mean matching time
	is lower in the racket-shaped graph
	than in the cycle with a chord,
	which may seem counterintuitive
	after observing that the former is obtained by
	\emph{removing} an edge from the latter.
	This observation is however consistent with
	the discussion of~\cite{CDFB20},
	which compared this phenomenon
	with Braess's paradox in road networks.

	\section{Conclusion} \label{sec:ccl}
	
	In this paper, we proved that
	a stochastic non-bipartite matching model
	is an order-independent loss queue.
	This equivalence allowed us
	to propose simpler proofs for existing results
	and to derive new results
	regarding stochastic matching models.
	In particular, we used
	the order-independent queue framework
	to give alternative proofs
	for the product-form stationary distribution
	and the stability condition
	derived in \cite[Theorem~1]{MBM21}.
	By adapting numerical methods
	developed in~\cite{SV15,C19,CBL21},
	we also derived new closed-form expressions
	for the normalization constant,
	the mean matching time,
	and other performance metrics.
	In turn, these formulas were applied
	to gain insight into
	the impact of parameters on performance.
	By analogy with queueing theory,
	we highlighted the importance
	of the so-called load parameters
	and studied the system behavior
	in the heavy-traffic regime,
	in which the load of a subset
	of classes becomes critical.
	The obtained formulas
	were also put in practice with numerical results.
	
	For future works,
	we would like to derive simpler formulas
	for the performance metrics.
	Indeed, even though our recursive formulas
	allow for a dynamic-programming approach,
	complexity remains exponential
	in the number of classes in general,
	which limits their applicability
	to models with only a few classes.
	In parallel, we would like to
	understand if and how these results
	can be adapted to analyze the performance
	of stochastic matching models
	with a bipartite compatibility graph,
	like the one studied
	in \cite{CKW09,AW12,BGM13,ABMW17,AKRW18},
	in which items arrive in pairs.
	Lastly, as mentioned before,
	we would like to explore the generalizations
	of the stochastic matching model
	made possible by the equivalence
	with order-independent loss queues.
	Yet another avenue for future work,
	which goes beyond the scope
	of stochastic matching models,
	consists of extending the heavy-traffic result
	proved for multi-class multi-server queues in~\cite{CBL21}
	to other order-independent (loss) queues,
	like those considered in~\cite[Section~4.3]{C19}.
	
	\paragraph{Acknowledgement}
	
	The author is grateful to
	Sem Borst for his valuable comments
	on an earlier draft of this paper
	and to Ellen Cardinaels
	for mentioning the reference~\cite{BMMR20}.
	The author thanks both of them
	for helpful discussions
	on the related work~\cite{CBL21}.
	The author is also grateful to
	the anonymous editor and reviewer
	for their valuable suggestions
	on the contents and exposition of the paper,
	and in particular for suggesting
	a simplification in Section~\ref{sec:oi}.


	\appendix
	
	\section*{Appendix: Proof of Proposition~\ref{prop:heavy}}
	
	As announced,
	the proof of each property
	relies on the corresponding expression
	derived in Section~\ref{sec:perf}.
	The key argument consists of observing that
	Equation~\eqref{eq:saturation}
	and Assumption~\ref{ass:stability}
	can be reformulated as follows:
	as $\rho \uparrow 1$,
	$\rho(\J)$ tends to one if $\J = \I$
	(as we have $\rho(\I) = \rho$)
	and to a constant strictly less than one
	if $\J \in \ind \setminus \{\I\}$.
	
	\begin{proof}[Proof of Property~\ref{item:heavy-piI}.]
		According to Proposition~\ref{prop:piI-rec}
		and Equation~\eqref{eq:piI-load},
		the stationary distribution
		of the set of unmatched classes
		can be rewritten as
		\begin{align} \label{eq:piI-ratio}
			\pi(\J)
			= \frac{\psi(\J)}{\sum_{\K \in \inde} \psi(\K)},
			\quad \J \in \inde,
		\end{align}
		where $\psi(\J)$ is given by $\psi(\emptyset) = 1$, and
		\begin{align*}
			\psi(\J)
			&= \frac{\rho(\J)}{1 - \rho(\J)}
			\left(
			\sum_{i \in \J} \frac{\alpha_i}{\alpha(\J)}
			\psi(\J \setminus \{i\})
			\right),
			\quad \J \in \ind.
		\end{align*}
		Since $\rho(\J)$ tends to one if $\J = \I$
		and to a constant strictly less than one
		if $\J \in \ind \setminus \{\I\}$,
		it follows that $\psi(\J)$ tends to
		infinity if $\J = \I$
		and to a finite constant
		if $\J \in \inde \setminus \{\I\}$.
		We conclude by inserting
		this into~\eqref{eq:piI-ratio}.
	\end{proof}
	
	\begin{proof}[Proof of Property~\ref{item:heavy-wi}]
		The result for the classes in~$\I$
		follows by inserting Property~\ref{item:heavy-piI}
		into~\eqref{eq:waiting}.
		The result for the classes in $\V \setminus \I$
		is then a consequence of~\eqref{eq:balance}
		and~\eqref{eq:saturation}.
	\end{proof}
	
	\begin{proof}[Proof of
		Properties~\ref{item:heavy-Li} and~\ref{item:heavy-ETi}]
		Consider a class $i \in \I$.
		According to Proposition~\ref{prop:LiI-rec}
		and Equation~\eqref{eq:piI-ratio},
		the expected number of class-$i$ items
		can be rewritten as
		\begin{align} \label{eq:Li-ratio}
			L_i = \frac
			{\sum_{\J \in \inde: i \in \J} L_i(\J) \psi(\J)}
			{\sum_{\J \in \inde} \psi(\J)},
		\end{align}
		where $L_i(\J) \psi(\J)$ is given by
		$L_i(\J) \psi(\J) = 0$ if $i \notin \J$, and
		\begin{align*}
			L_i(\J) \psi(\J)
			= \frac{\rho(\J)}{1 - \rho(\J)}
			\begin{aligned}[t]
				\Bigg(
				&\frac{\alpha_i}{\alpha(\J)} \psi(\J)
				+ \frac{\alpha_i}{\alpha(\J)} \psi(\J \setminus \{i\}) \\
				&+ \sum_{j \in \J \setminus \{i\}}
				\frac{\alpha_j}{\alpha(\J)}
				L_i(\J \setminus \{j\})
				\psi(\J \setminus \{j\})
				\Bigg),
			\end{aligned}
		\end{align*}
		if $i \in \J$.
		By following a similar reasoning
		as in the proof
		of Property~\ref{item:heavy-piI}
		and recalling that
		$\rho(\J)$ tends to one if $\J = \I$
		and to a constant strictly less than one
		if $\J \in \ind \setminus \{\I\}$, we obtain
		\begin{align*}
			L_i \underset{\rho \uparrow 1}{\sim} \frac
			{ \frac{\rho}{1 - \rho}
				\frac{\alpha_i}{\alpha(\I)} \psi(\I) }
			{\psi(\I)}.
		\end{align*}
		After simplification, we conclude
		by observing that~\eqref{eq:scaling}
		implies $\alpha_i / \alpha(\I) = p_i$.
		If we consider instead $i \in \V \setminus \I$,
		a similar reasoning shows that
		all terms in the numerator
		of~\eqref{eq:Li-ratio}
		have a finite limit,
		while its denominator still tends to infinity,
		so that $L_i$ tends to zero
		as $\rho \uparrow 1$.
		The limiting result for $L$ follows by summation,
		and that for
		$\esp{T_i}$, $i \in \V$,
		follows by applying Little's law
		and simplifying the result.
	\end{proof}
	
	\begin{proof}[Proof of Property~\ref{item:heavy-X}]
		Consider the random vector
		$S = (S_1, S_2, \ldots, S_N)$ defined by
		\begin{align*}
			S_i = \begin{cases}
				(1 - \rho) X_i,
				&i \in \I, \\
				X_i,
				& i \in \V \setminus \I.
			\end{cases}
		\end{align*}
		The Laplace-Stieltjes transform
		of this random vector is defined,
		for each $u \in \R_+^N$, by
		\begin{align*}
			\varphi_S(u)
			&= \esp{\prod_{i \in \V} e^{- u_i S_i}}
			= \esp{
				\left(
				\prod_{i \in \I} \e^{- u_i (1 - \rho) X_i}
				\right) \left(
				\prod_{i \in \V \setminus \I} \e^{- u_i X_i}
				\right)
			}
			= \esp{\prod_{i \in \V} {z_i}^{X_i}}
			= g_X(z),
		\end{align*}
		where the vector
		$z = (z_1, z_2, \ldots, z_N)$
		is given by
		$z_i = \e^{- u_i (1 - \rho)}$ for $i \in \I$
		and $z_i = \e^{- u_i}$ for $i \in \V \setminus \I$.
		We will use Proposition~\ref{prop:pgf-rec}
		to show that
		this Laplace-Stieltjes transform satisfies
		\begin{align} \label{eq:limit-X}
			\varphi_S(u)
			\xrightarrow[\rho \uparrow 1]{}
			\frac1{1 + \sum_{i \in \I} p_i u_i}.
		\end{align}
		One can verify that this is also
		the Laplace-Stieltjes transform
		of a random vector
		$\tilde S = (\tilde S_1, \tilde S_2,
		\ldots, \tilde S_N)$
		such that
		$\tilde S_i = \expo(1) p_i$
		for each $i \in \I$
		and $\tilde S_i = 0$
		for each $i \in \V \setminus \I$.
		According to \cite[Section 13.1, Theorem 2]{F66},
		this suffices to conclude that
		$S$ weakly converges to $\tilde S$
		as $\rho \uparrow 1$.
		
		By following a similar approach
		as in the proof of
		Properties~\ref{item:heavy-piI}
		and~\ref{item:heavy-Li}
		and recalling that
		$z_i \uparrow 1$ if $i \in \I$
		and $z_i = e^{- u_i} < 1$
		if $i \in \V \setminus \I$,
		we obtain
		\begin{align*}
			\varphi_S(u)
			\underset{\rho \uparrow 1}{\sim} \frac
			{\psi_{z \alpha, \alpha}(\I)}
			{\psi_{\alpha, \alpha}(\I)}.
		\end{align*}
		Using~\eqref{eq:pilambda-rec},
		we can rewrite the right-hand side
		of this equivalence as
		\begin{align} \label{eq:dominant}
			\frac
			{\psi_{z \alpha, \alpha}(\I)}
			{\psi_{\alpha, \alpha}(\I)}
			= \frac{ \displaystyle
				\frac
				{\sum_{i \in \I} z_i \alpha_i
					\psi_{z \alpha, \alpha}(\I \setminus \{i\})}
				{\alpha(\V \setminus \I) - (z \alpha)(\I)}
			}{ \displaystyle
				\frac
				{\sum_{i \in \I} \alpha_i
					\psi_{\alpha, \alpha}(\I \setminus \{i\})}
				{\alpha(\V \setminus \I) - \alpha(\I)}
			}
			= \frac
			{\sum_{i \in \I} z_i \alpha_i
				\psi_{z \alpha, \alpha}(\I \setminus \{i\})}
			{\sum_{i \in \I} \alpha_i
				\psi_{\alpha, \alpha}(\I \setminus \{i\})}
			\cdot
			\frac
			{1 - \frac{\alpha(\I)}{\alpha(\V \setminus \I)}}
			{1 - \frac{(z \alpha)(\I)}{\alpha(\V \setminus \I)}}.
		\end{align}
		As $\rho \uparrow 1$,
		the first factor tends to~$1$
		because $z_i \uparrow 1$ for each $i \in \I$,
		while we can show with a Taylor series expansion
		that the second factor tends to
		$(1 + \sum_{i \in \I} p_i u_i)^{-1}$.
		This proves~\eqref{eq:limit-X}
		and completes the proof.
	\end{proof}
	
	\begin{proof}[Proof of Property~\ref{item:heavy-Ti}]
		We will show that,
		for each $i \in \I$, we have
		\begin{align} \label{eq:limit-Ti}
			\lim_{\rho \uparrow 1}
			\varphi_{(1 - \rho) T_i}(u)
			\xrightarrow[\rho \uparrow 1]{}
			\frac{\frac12}{\frac12 + u},
			\quad i \in \I,
			\quad \qquad
			\lim_{\rho \uparrow 1}
			\varphi_{T_i}(u)
			\xrightarrow[\rho \uparrow 1]{}
			1,
			\quad i \in \V \setminus \I.
		\end{align}
		We recognize the Laplace-Stieltjes transform
		of an exponentially-distributed random variable
		with rate $\frac12$ for $i \in \I$
		and the degenerate Laplace-Stieltjes transform
		of a constant random variable equal to~$0$
		for $i \in \V \setminus \I$.
		The conclusion of the proposition
		then again follows
		from \cite[Section 13.1, Theorem 2]{F66}.
		
		First let $i \in \I$.
		According to Proposition~\ref{prop:lst-rec}, we have
		\begin{align*}
			\varphi_{(1 - \rho) T_i}(u)
			= \varphi_{T_i}((1 - \rho) u)
			= g_X(z),
		\end{align*}
		where $z = (z_1, z_2, \ldots, z_N)$
		is given by
		$z_i = 1 - (1 - \rho) \frac{u}{\alpha_i}$
		and $z_j = 1$ for each $j \in \V \setminus \{i\}$.
		Since $z_i \uparrow 1$ as $\rho \uparrow 1$,
		and by following a similar reasoning
		as in the proof of Property~\ref{item:heavy-X},
		we obtain
		\begin{align*}
			\varphi_{(1 - \rho) T_i}(u)
			\underset{\rho \uparrow 1}{\sim} \frac
			{\psi_{z \alpha, \alpha}(\I)}
			{\psi_{\alpha, \alpha}(\I)}.
		\end{align*}
		The right-hand side of this equivalence
		is again given by~\eqref{eq:dominant}.
		By using the definition of $z$,
		we can show that the second factor
		in~\eqref{eq:dominant}
		simplifies to
		$\alpha(\V \setminus \I) / (\alpha(\V \setminus \I) + u)$,
		which tends to $\frac12 / (\frac12 + u)$
		according to~\eqref{eq:limit-Ti}.
		Since we also know that
		the first factor in~\eqref{eq:dominant} tends to~$1$,
		this proves~\eqref{eq:limit-Ti}.
		The proof for $i \in \V \setminus \I$ is similar,
		the main difference being that we have directly
		$\psi_{z \alpha, \alpha}(\I)
		= \psi_{\alpha, \alpha}(\I)$
		because $i \notin \I$.
	\end{proof}
	
\end{document}